\documentclass{amsart}
\usepackage[all]{xy}
\usepackage[dvipsnames]{xcolor}
\usepackage{amssymb}
\usepackage{enumerate}
\usepackage{mathabx}
\usepackage{tikz,calligra,mathrsfs}
\usetikzlibrary{matrix,arrows,decorations.pathmorphing}
\usepackage[all]{xy}
\usepackage{graphicx}
\usepackage{caption}
\usetikzlibrary{positioning}
\let\cal\mathcal

\def\Fscr{{\cal F}}
\def\Gscr{{\cal G}}

\def\Lscr{{\cal L}}
\def\Mscr{{\cal M}}
\def\Nscr{{\cal N}}
\def\Oscr{{\cal O}}

\def\Rscr{{\cal R}}

\def\Wscr{{\cal W}}

\def\Bimod{\operatorname{Bimod}}
\def\BIMOD{\operatorname{BIMOD}}

\def\Gr{\operatorname{Gr}}

\def\QGr{\operatorname{QGr}}

\def\Qcoh{\operatorname{Qcoh}}

\def\Ext{\operatorname {Ext}}
\def\Hom{\operatorname {Hom}}

\def\RHom{\operatorname {RHom}}

\def\ker{\operatorname {ker}}

\def\gkdim{\operatorname {GK dim}}

\def\pd{\operatorname {pd}}

\DeclareMathOperator{\Proj}{Proj}

\DeclareMathOperator{\Tors}{Tors}

\DeclareMathOperator{\Aut}{Aut}

\def\Frac{\operatorname{Frac}}

\newcommand{\hookdownarrow}{\mathrel{\rotatebox[origin=c]{-90}{$\hookrightarrow$}}}
\newtheorem{lemma}{Lemma}[section]
\newtheorem{proposition}[lemma]{Proposition}
\newtheorem{theorem}[lemma]{Theorem}
\newtheorem{corollary}[lemma]{Corollary}

\newtheorem{convention}[lemma]{Convention}

\newtheorem*{theorem*}{Theorem}

\theoremstyle{definition}

\newtheorem{definition}[lemma]{Definition}

{

}

\theoremstyle{remark}

\newtheorem{remark}[lemma]{Remark}

\newtheorem{notation}{Notation}

\newdimen\uboxsep \uboxsep=1ex
\def\uboxn#1{\vtop to 0pt{\hrule height 0pt depth 0pt\vskip\uboxsep
\hbox to 0pt{\hss #1\hss}\vss}}

\def\uboxs#1{\vbox to 0pt{\vss\hbox to 0pt{\hss #1\hss}
\vskip\uboxsep\hrule height 0pt depth 0pt}}

\let\oldmarginpar\marginpar
\def\marginpar#1{\oldmarginpar{\tiny\textcolor{red}{#1}}}

\newcommand\blfootnote[1]{%
  \begingroup
  \renewcommand\thefootnote{}\footnote{#1}%
  \addtocounter{footnote}{-1}%
  \endgroup
}

\author{Dennis Presotto}
\title[Symmetric noncommutative birational transformations.]{Symmetric noncommutative birational transformations.}

\begin{document}
\begin{abstract}
In \cite{PresVdB} certain birational transformations were constructed between the noncommutative schemes associated to quadratic and cubic three dimensional Sklyanin algebras.
In the current paper we consider the inverse birational transformations and show that they are of the same type. 
Moreover we extend everything to the $\mathbb{Z}$-algebras context, which allows us to incorporate the noncommutative quadrics introduced by Van den Bergh in \cite{VdB38}.
\end{abstract}
\maketitle
\blfootnote{The author was supported by a Ph.D. fellowship of the Research Foundation Flanders (FWO).}
\tableofcontents

\section{Introduction}

Throughout this paper $k$ will be an algebraically closed field and all rings will be algebras over $k$. 
%Inspired by Serre's Theorem one can associate 
Following tradition \cite{artinzhang} we associate a noncommutative projective scheme $X = \Proj(A)$ to a connected graded $k$-algebra $A=k \oplus A_1 \oplus \ldots$ which is generated by $A_1$. $\Proj(A)$ is defined via its category of ``quasicoherent sheaves'':
\[ \Qcoh( X) := \QGr(A) = \Gr(A)/ \Tors(A) \]
where $\Gr(A)$ is the category of graded right $A$-modules and $\Tors(A)$ is the full subcategory of torsion $A$-modules, i.e. those graded right
$A$-modules that have locally right bounded grading \cite{artinzhang}. 
%(It is customary to write $X$ for both $\Proj(A)$ as $\QGr(A)$.) When $A$ has sufficiently nice homological and growth properties one can think of $X$ as being a noncommutative variety over $k$ in stead of just a noncommutative scheme. A nice class of algebras having such sufficiently nice properties are the so called Artin-Schelter regular algebras \cite{artinschelter}. Recall that a connected graded algebra\footnote{Although this is not necessary for the definition of an AS-regular algebra, we will only consider the case where $A$ is generated in degree 1.} $A$ is called AS-regular if it satisfies the following conditions
A non-trivial class of noncommutative surfaces is given by three dimensional Artin-Schelter regular algebras as defined in \cite{artinschelter}. Recall that a connected graded algebra $A$ is called AS-regular of dimension $d$ if it satisfies the following conditions
\begin{enumerate}[(i)]
\item The Hilbert series $h_n(A) := \dim_k(A_n)$ is bounded by a polynomial in $n$.
\item $A$ has finite global dimension $d$
\item $A$ has the Gorenstein property with respect to $d$
\end{enumerate}
Such algebras with $d=3$ and which are generated in degree 1 are classified in \cite{artinschelter, ATV1}. There are two possibilities for the number of generators and relations:
\begin{enumerate}[(i)]
\item $A$ is generated by three elements satisfying three quadratic
  relations (the ``quadratic case'').  In this case $A$ has Hilbert
  series $1/(1-t)^3$, i.e.\ the same Hilbert series as a polynomial
  ring in three variables. Therefore we think of $\Proj(A)$ as being a noncommutative $\mathbb{P}^2$.
\item $A$ is generated by two elements satisfying two cubic relations
(the ``cubic case''). In this case $A$ has Hilbert series $1/(1-t)^2(1-t^2)$ and we can think of $\Proj(A)$ as being a noncommutative $\mathbb{P}^1 \times \mathbb{P}^1$. (The rationale for this is explained in \cite{VdB38}.)
\end{enumerate}
We write $(r,s)$ for the number of generators of $A$ and the degrees of the relations. Thus $(r,s)=(3,2)$ or $(2,3)$ depending on whether $A$ is quadratic or cubic.\\
The classification of three-dimensional AS-regular algebras $A$ is in terms of suitable geometric data $(Y,\Lscr,\sigma)$ where $Y$ is a $k$-scheme, $\sigma$ is an automorphism of $Y$ and $\Lscr$ is a line bundle on $Y$. More precisely: starting from $(Y,\Lscr,\sigma)$ one can construct a twisted homogeneous coordinate ring $B(Y,\Lscr,\sigma)$ (see for example \cite{AV}). This is a connected graded $k$-algebra with $r := \dim_k(B_1)$ equal to 2 or 3. Setting $s=5-r$ as above, we can then find the AS-regular algebra $A(Y,\Lscr,\sigma)$ by dropping all relations in degree $s+1$ and higher. In particular there is a surjective morphism $A(Y,\Lscr,\sigma) \rightarrow B(Y,\Lscr,\sigma)$ giving rise to an inclusion 
\begin{eqnarray}
\label{eq:commutativecurve} \QGr(B) \hookrightarrow \QGr(A)
\end{eqnarray}
Moreover there is an equivalence of categories $\QGr(B) \cong \Qcoh(Y)$ (see for example \cite{AV}). Therefore one often says there is a commutative curve $Y$ contained inside the noncommutative surface $X$.

%We are mainly interested in the case where $A$ is a Sklyanin algebra, this corresponds to $Y$ being an elliptic curve and $\sigma$ being given by translation.
Below we say that $A$ is a (quadratic or cubic) Sklyanin algebra if $Y$ is smooth and $\sigma$ is a translation.
%In this case by $(\ref{eq:commutativecurve})$ and the equivalence of categories $\QGr(B) \cong \Qcoh(Y)$ (see for example \cite{AV}) we can say that there is a commutative curve $Y$ contained inside the noncommutative surface $X$. It is also known that there is a 1-1-correspondence between the points of $Y$ and the pointmodules of $A$.
\\

%One of the goals of noncommutative geometry is to classify these noncommutative planes upto ``birational equivalence''. 
Inspired by the commutative case it makes sense to expect that in a suitable sense noncommutative $\mathbb{P}^2$s are birationally equivalent to noncommutative $\mathbb{P}^1 \times \mathbb{P}^1$s. Such birational equivalences were constructed in \cite{RogSierStaf} and \cite{PresVdB}.

In \cite{PresVdB} we provide a noncommutative version of the standard birational transformation $\mathbb{P}^1 \times \mathbb{P}^1 \dashrightarrow \mathbb{P}^2$  by showing that for each cubic Sklyanin algebra $A$ there exists a quadratic Sklyanin algebra $A'$ and an inclusion 
\begin{eqnarray}
\label{eq:valuesvw} \check{A'} \hookrightarrow \check{A}^{(2)}
\end{eqnarray}
 where $\check{A}$ and $\check{A'}$ are the associated $\mathbb{Z}$-algebras. (See \S \ref{sec:Z-algebras} for more on $\mathbb{Z}$-algebras.) We then check that this inclusion gives rise to an isomorphism of the function fields $\Frac_0(A') \cong \Frac_0(A)$, a result which was already announced in \cite{VdBSt} and \cite{Sierratalk}. 
In \cite{PresVdB} we also provide a noncommutative version of the Cremona transform $\mathbb{P}^2 \dashrightarrow \mathbb{P}^2$ by showing that for each quadratic Sklyanin algebra $A$ there exists a quadratic Sklyanin algebra $A'$ and an inclusion 
\begin{eqnarray}
\label{eq:valuesvwquadratic} \check{A'} \hookrightarrow \check{A}^{(2)}
\end{eqnarray}
The construction was based on the choice of 3 non-collinear points $p,q,r$ on $Y$ and if $A=A(Y,\Lscr,\psi)$ then $A'=A(Y,\Lscr \otimes \psi^*\Lscr \otimes \Oscr_Y(-p-q-r), \psi)$. \\

In \S \ref{sec:PresVdB} we recapitulate the construction of \eqref{eq:valuesvw} and \eqref{eq:valuesvwquadratic}.\\

In \S \ref{sec:addendum} we extend the above to the level of $\mathbb{Z}$-algebras (see \S \ref{sec:Z-algebras} and Remark \ref{rem:defsklyaninquadric} for the appropriate definitions) and prove the following:
\begin{theorem*}[Theorem \ref{thm:firstextension}]
Let $A$ be a cubic Sklyanin $\mathbb{Z}$-algebra. Then there is a quadratic Sklyanin $\mathbb{Z}$-algebra $A'$ and an inclusion  $A' \hookrightarrow A^{(2)}$ inducing an isomorphism between their function fields. This inclusion is constructed with respect to a point $p$ on $Y$ such that if $A = A(Y,\Lscr_0,\Lscr_1,\Lscr_2)$ then $A'=A(Y,\Lscr_0 \otimes \Lscr_1 \otimes \Oscr_Y(-p), \Lscr_2 \otimes \Lscr_3 \otimes \Oscr_Y(-\tau^{-1}p)$.
\end{theorem*}
\begin{remark}
The existence of the function field a cubic Sklyanin $\mathbb{Z}$-algebra (or more generally a ``quadric'') is proven in Appendix \ref{sec:appendix}.
\end{remark}
In \S \ref{sec:addendum} we also provide noncommutative versions of the inverse birational transformation $\mathbb{P}^2 \dashrightarrow \mathbb{P}^1 \times \mathbb{P}^1$ as follows:
\begin{theorem*}[Theorem \ref{thm:secondextension}]
Let $A$ be a quadratic Sklyanin $\mathbb{Z}$-algebra, then there is a cubic Sklyanin $\mathbb{Z}$-algebra $A'$ and an inclusion $A' \hookrightarrow A$ inducing an isomorphism between their function fields. This inclusion is constructed with respect to points $p,q$ on $Y$ such that if $A = A(Y,\Lscr,\psi)$ then $A'=A(Y,\Lscr \otimes \Oscr_Y(-p), \psi^* \Lscr \otimes \Oscr_Y(-q), \psi^{*2} \Lscr \otimes \Oscr_Y(-\psi^{-3}p))$.
\end{theorem*}

In \S \ref{sec:cremona} we show (Theorem \ref{thm:cremona}) that, modulo some technical hypothesis, the noncommutative Cremona as in \eqref{eq:valuesvwquadratic} factors through the noncommutative $\mathbb{P}^2 \dashrightarrow \mathbb{P}^1 \times \mathbb{P}^1$ and $\mathbb{P}^1 \times \mathbb{P}^1 \dashrightarrow \mathbb{P}^2$. As such all of these are examples of ``quadratic transforms'', a more general type of noncommutative birational transformations introduced in \S \ref{sec:cremona}. \\

In the last sections we show that quadratic transforms are invertible in the following sense (for simplicity we omit some technical hypotheses):
\begin{theorem*}[Theorem \ref{thm:inversemain}]
Let $\gamma: A' \hookrightarrow A^{(w)}$ be a quadratic transform. Then there exists a quadratic transform $\delta: A \hookrightarrow A'^{(v)}$ such that the compositions $\gamma \circ \delta: A \hookrightarrow A^{(vw)}$ and $\delta \circ \gamma: A' \hookrightarrow A'^{(vw)}$ induce the identity map on the function fields.
\end{theorem*}
\begin{remark}
In the above theorem the possible values for $w$ are determined by the definition of a quadratic transform. Moreover $v$ can be chosen as a function of $w$ and $\frac{v}{w} \in \left \{ \frac{1}{2}, 1, 2 \right \}$.
\end{remark}
In \S \ref{sec:innermorphisms} we show that it suffices to prove this theorem in case $\gamma$ is as in Theorem \ref{thm:firstextension} or Theorem \ref{thm:secondextension}.
Both of these cases are covered in \S \ref{sec:invertibletransformquadric}. The proof is quite technical and uses a $\mathbb{Z}^2$-algebra which in a certain sense glues $A$ and $A'$. On the other hand the geometric picture is rather simple. For example if $A=A(Y,\Lscr,\psi)$ is a quadratic Sklyanin algebra and $\gamma: A' \hookrightarrow A$ is constructed with respect to points $p,q$ on $Y$ as in Theorem \ref{thm:secondextension}, then the inverse $\delta: A \hookrightarrow A^{\prime (2)}$ is constructed with respect to $p'$ where $p+q+p' \sim [\psi_* \Lscr]$ in the Picard group of $Y$. Conversely if $A=A(Y,(\Lscr_i)_i)$ is a cubic Sklyanin $\mathbb{Z}$-algebra and $\gamma: A' \hookrightarrow A^{(2)}$ is constructed with respect to a point $p$ on $Y$ as in Theorem \ref{thm:firstextension}, then the inverse $\delta: A \hookrightarrow A'$ is constructed with respect to $p', q'$ where $p+ \tau q' \sim [\Lscr_0]$ and $p + p' \sim [\Lscr_1]$ in the Picard group of $Y$.

%We recommend the reader to start by only reading sections \ref{sec:Z-algebras}, \ref{sec:PresVdB}, \ref{sec:quadratictransforms} and \ref{sec:invertibletransformquadratic}. This way the reader first gets an overview of the noncommutative version of the Cremona transform $\mathbb{P}^2 \dashrightarrow \mathbb{P}^2$ as in \cite{PresVdB} and its inverse as in \S \ref{sec:invertibletransformquadratic}. Afterwards the reader can go to sections \ref{sec:addendum}, \ref{sec:invertibletransformquadric} and the appendix to see how one can include quadrics in this picture.

\section{Acknowledgements}
The author wishes to thank Michel Van den Bergh for providing many interesting ideas and for reading through the results multiple times.
The author is also very grateful to Susan J. Sierra for inviting him to the University of Edinburgh and for posing the question about the nature of the inverses to the birational transformations considered in \cite{PresVdB}.
This question directly lead to the current paper.
\section{$\mathbb{Z}$-algebras}
\label{sec:Z-algebras}

In this section we recall some definitions and facts on $\mathbb{Z}$-algebras. We refer the reader to \cite{Sierra} or sections 3 and 4 of \cite{VdB38} for a more thorough introduction. Recall
that a $\mathbb{Z}$-algebra is defined as an algebra $R$ (without
unit) with a decomposition 
\[ \displaystyle R = \bigoplus_{(m,n) \in \mathbb{Z}^2} R_{m,n} \]
such that addition is degree-wise and multiplication satisfies $R_{m,n}R_{n,j} \subset R_{m,j}$ and $R_{m,n}R_{i,j}=0$ if $n \neq i$. Moreover there are local units $e_n\in R_{n,n}$ such that for each $x \in R_{m,n}: e_m x = x = x e_n$. 
\begin{notation}
If $A$ is a graded algebra, then it gives rise to a
$\mathbb{Z}$-algebra $\check{A}$ via $\check{A}_{m,n} = A_{n-m}$.
\end{notation} 
In particular the notion of a $\mathbb{Z}$-algebra is a generalization of a ($\mathbb{Z}$)-graded algebra. Based on this, most graded notions have a natural $\mathbb{Z}$-algebra counterpart. For example we say that a $\mathbb{Z}$-algebra $R$ is positively graded if $R_{m,n}=0$ for $m > n$. \\

A graded $R$-module is an $R$-module $M$ together with a decomposition $M = \oplus_n M_n$ such that the $R$-action on $M$ satisfies $M_m R_{m,n} \subset M_n$ and $M_m R_{i,n} =0$ if $i \neq m$. The category of graded $R$-modules is denoted $\Gr(R)$ and similar to the graded case we use the notation $\QGr(R) := \Gr(R) / \Tors(R)$.

\begin{remark} \label{rem:OX}
If $A$ is a graded algebra, then obviously $\Gr(A)=\Gr(\check{A})$ by identifying $A(n) \in \Gr(A)$ with $e_{-n}\check{A} \in \Gr(\check{A})$. Similarly $\QGr(A)=\QGr(\check{A})$.
\end{remark}

\begin{definition}
Let $R$ be a $\mathbb{Z}$-algebra. Then for each $n \in \mathbb{Z}$ we define $R(n)$ by setting $\left(R(n) \right)_{i,j} = R_{i+n,j+n}$ with obvious multiplication. We say $R$ is $n$-periodic if there is a $\mathbb{Z}$-algebra-isomorphism $R \cong R(n)$.
\end{definition}

\begin{lemma}
Let $R$ be a $\mathbb{Z}$-algebra. Then there exists a graded algebra $A$ such that $R = \check{A}$ if and only if $R$ is 1-periodic.
\end{lemma}
\begin{proof}
The ``only if'' part is obvious from the definition of $\check{A}$. The ``if'' part is proven in \cite[Lemma 3.4]{VdB38}.
\end{proof}

\begin{remark} \label{rem:quotientfield}
From a categorical point of view a $\mathbb{Z}$-algebra $R$ is nothing but a $k$-linear category $\Rscr$ whose objects are given by the integers. The homogeneous elements of the algebra then correspond to morphisms between two such integers via $\Hom_\Rscr(-j,-i) = R_{i,j}$ and multiplication in $R$ corresponds to composition of morphisms in $\Rscr$. 

%This observation allows us to introduce the $\mathbb{Z}$-field of fractions $Q$ or $R$ as the $\mathbb{Z}$-algebra corresponding to $\Qscr := \Rscr [W^{-1} ]$ where $W$ is the set of nonzero morphism in $\Rscr$, i.e. the nonzero homogeneous elements in $R$. This $\mathbb{Z}$-field of fractions $Q$ only exists when \\ \\

%NOTE: SOMEWHERE WE NEED TO EXPLAIN WHY AS REGULAR Z ALGEBRAS ALLOW Z FIELDS OF FRACTIONS. THE ARGUMENT IS AS FOLLOWS: QUADRATIC AS REGULAR ALGEBRAS ALLOW IT AS THEY COME FROM GRADED ALGEBRAS (see for example \cite[Corollary 8.4.6.]{nastasescumethods}). HOPEFULLY THERE IS A SIMILAR ARGUMENT FOR QUADRICS? OTHERWISE WE CAN USE THE FOLLOWING: ANY SKLYANIN QUADRIC IS CONTAINED IN A QUADRATIC AS REGULAR Z ALGEBRA, HENCE HAS NO ZERO DIVISORS. SIMILARLY THE PROOF OF ISOMORPHIC FUNCTION FIELDS SHOW THAT THE RIGHT ORE CONDITION IS SATISFIED AS QUADRICS CONTAIN A QUADRATIC AS REGULAR Z ALGEBRA.

\end{remark}

\subsection{AS-regular $\mathbb{Z}$-algebras}

\begin{definition}
Let $R$ be a $\mathbb{Z}$-algebra, then $R$ is said to be connected, if it is positively graded, $\dim_k(R_{m,n})<\infty$ for each $m,n$ and $R_{m,m} \cong k$ for all $m$. We say $R$ is generated in degree 1 if $R_{m,m+1}R_{m+1,n} = R_{m,n}$ holds for all $m<n$. If $R$ is a connected $\mathbb{Z}$-algebra, generated in degree 1, then we denote $S_{n,R} = e_nR/ (e_n R)_{\geq n+1}$. I.e. $S_{n,R}$ is the unique $R$-module concentrated in degree $n$ where it is equal to the base field $k$.
\end{definition}

We can now give the definition of an AS-regular $\mathbb{Z}$-algebra as in \cite{VdB38}
\begin{definition}
A $\mathbb{Z}$-algebra $R$ over $k$ is said to be AS-regular if the following conditions are satisfied:
\begin{enumerate}
\item $R$ is connected and generated in degree 1
\item $\dim_k(R_{m,n})$ is bounded by a polynomial in $n-m$
\item The projective dimension of $S_{n,R}$ is finite and bounded by a number independent of $n$
\item $\displaystyle \forall n \in \mathbb{N}: \sum_{i,j} \dim_k \left( \Ext_{\Gr(R)}^i(S_{j,R},e_nR) \right) =1$ (the ``Gorenstein condition'')
\end{enumerate}
\end{definition}
It is immediate that if a graded algebra $A$ is AS-regular, then $\check{A}$ is AS-regular in the above sense.

$\mathbb{Z}$-algebra analogues of three dimensional quadratic and cubic AS-regular algebras were classified in \cite{VdB38}, it is shown in loc.sit. that every quadratic AS-regular $\mathbb{Z}$-algebra is of the form $\check{A}$ for some quadratic AS-regular algebra $A$. However most cubic AS-regular $\mathbb{Z}$-algebras are not 1-periodic. To distinguish cubic AS-regular algebras from the more general cubic AS-regular $\mathbb{Z}$-algebras, one often refers to the latter as \emph{quadrics}.

Similar to the graded case, the classification of three-dimensional quadratic and cubic AS-regular $\mathbb{Z}$-algebras in terms of geometric data $(Y, ( \Lscr_i )_{i \in \mathbb{Z}})$ where $Y$ is a $k$-scheme and $( \Lscr_i )_{i \in \mathbb{Z}}$ is an elliptic helix of line bundles on $Y$ (see \cite{bondalpolishchuk} for more information on helices). Starting from the geometric data one first constructs a $\mathbb{Z}$-algebra analogue  $B=B(Y,(\Lscr_i)_i)$ of the twisted homogeneous coordinate ring:
\[ B_{i,j} := \begin{cases} \Gamma(Y, \Lscr_i \otimes \Lscr_{i+1} \otimes \ldots \Lscr_{j-1}) & \textrm{if $i \leq j$} \\ 0 & \textrm{ if $i > j$} \end{cases} \]

 Again there is an equivalence of categories.
\begin{eqnarray} \label{eq:YisB} \Qcoh(Y) \cong \QGr(B): \Fscr \mapsto \bigoplus_{i \geq 0} \Gamma(Y, \Fscr \otimes \Lscr_0 \otimes \Lscr_1 \otimes \ldots \Lscr_{i-1}) \end{eqnarray}
(see \cite[Corollary 5.5.9]{VdB38}). $r = \dim_k(B_{i,i+1})$ does not depend on $i$ and equals 3 (in the quadratic case) or 2 (in the cubic case). $A(Y,(\Lscr_i)_i)$ is then obtained from $B$ by only preserving the relations in degree $(i,i+s)$ for $s=5-r$. It is shown that $\dim_k( A_{i,i+n})$ does not depend on $i$; it hence makes sense to write $h(n) := \dim_k(A_{i,i+n})$ and one checks that $h(n)$ coincides with the Hilbert series in the graded case (even thought AS-regular $\mathbb{Z}$-algebras need not be 1-periodic).

\subsection{$\mathbb{Z}$-domains and $\mathbb{Z}$-fields of fractions}

In this section we give the natural generalizations of ``domain'' and ``field of fractions'' for $\mathbb{Z}$-algebras. Among other generalizations, these notions can also be found in \cite[\S 2]{ChanNyman}.

\begin{definition}
Let $R$ be a $\mathbb{Z}$-algebra. Then we say that $R$ is a $\mathbb{Z}$-domain if the following condition is satisfied:
\[ \forall i,j,k \in \mathbb{Z}, \forall x \in R_{i,j}, \forall y \in R_{j,k}: xy=0 \Rightarrow x=0 \ \vee \ y=0 \]
\end{definition}
It is known that three dimensional AS-regular algebras are domains (\cite[Theorem 3.9]{ATV2} and \cite[Theorem 8.1]{ATV2}). We extend this result to the level of quadrics and show:
\begin{theorem} \label{thm:quadricdomain}
Let $A = A(Y, ( \Lscr_i )_i)$ be a quadric. Then $A$ is a $\mathbb{Z}$-domain.
\end{theorem}
\begin{proof}
The proof of this theorem is postponed to Appendix \ref{sec:appendix}.
\end{proof}

Let $R$ be a $\mathbb{Z}$-algebra and $\Rscr$ the associated category as in Remark \ref{rem:quotientfield}. Let $W$ be a collection of homogeneous elements and let $\Wscr$ be the corresponding collection of morphisms. We then say that $R$ is localizable at $W$ if %it is possible to invert all morphisms in $\Wscr$, i.e. when 
$(\Rscr,\Wscr)$ admits a calculus of fractions (see for example \cite{GabrielZisman}). In this case we define $R[W^{-1}]$ as the $\mathbb{Z}$-algebra associated to $\Rscr[\Wscr^{-1}]$.\\

Using the theory of (right) fractions for a category one easily checks that the following definition makes sense.
\begin{definition} \label{def:Zfieldoffractions}
Let $R$ be a $\mathbb{Z}$-domain, then $R$ admits a $\mathbb{Z}$-field of (right) fractions if the following condition is satisfied:
\[ \forall r \in R_{l,i}, s \in R_{l,j} \setminus \{0\}: \exists n \in \mathbb{Z}: \exists r' \in R_{j,n}, s' \in R_{i,n} \setminus \{0\}: rs' = sr' \]
The elements of $\Frac(R)_{i,j}$ are equivalence classes of couples $(r,s)$ where $r \in R_{i,l}, s \in R_{j,l}\setminus \{ 0 \}$ for some $l \in \mathbb{Z}$. The equivalence relation is given by 
\begin{eqnarray*} &\forall l_1, l_2 \in \mathbb{Z}, \forall r_1 \in R_{i,l_1}, r_2 \in R_{i,l_2}, s_1 \in R_{j,l_1}\setminus \{ 0 \}, s_2 \in R_{j,l_2} \setminus \{ 0 \}:& \\
&(r_1,s_1) \sim (r_2,s_2)& \\ &\Updownarrow& \\ & \exists l_3 \in \mathbb{Z}, \exists x \in R_{l_1,l_3} \setminus \{0\}, \exists y \in R_{l_2,l_3} \setminus \{0 \}: r_1 x = r_2 y \textrm{ and } s_1 x = s_2 y& \end{eqnarray*}
\end{definition}
The following is obvious from the definition:
\begin{proposition} \label{prp:gradedZfractions}
Let $A$ be a graded domain. Then
\[ A \textrm{ admits a (graded) field of (right) fractions} \Leftrightarrow \check{A} \textrm{ admits a $\mathbb{Z}$-field of (right) fractions} \]
Moreover in this case $\widecheck{Frac(A)} = Frac(\check{A})$
\end{proposition}
Recall that graded domains admit fields of (right-)fractions when they are graded (right-)noetherian or when they have subexponential growth. A similar result can be found in \cite{ChanNyman}:
\begin{proposition} \label{prp:channyman}
Let $R$ be a $\mathbb{Z}$-domain such that each $e_i R$ is a uniform module (i.e. for all nonzero $M,N \subset e_i R: M \cap N \neq 0$). Then $R$ admits a $\mathbb{Z}$-field of (right) fractions.
\end{proposition}
\begin{proof}
We need to show that for all $r \in R_{l,i}, s \in R_{l,j} \setminus \{0\}$ there exists an $n \in \mathbb{Z}$ and elements $r' \in R_{j,n}, s' \in R_{i,n} \setminus \{0\}$ such that $rs' = sr'$. If $r=0$ then it suffices to take $r'=0$. If $r \neq 0$ the existence of $r'$ and $s'$ follows from the fact that $rR$ and $sR$ are nonzero submodules of the uniform module $e_lR$. 
\end{proof}
Similarly the following will be shown in Appendix \ref{sec:appendix}:

\begin{theorem} \label{thm:quadricsfunctionfield}
Let $A$ be a quadric. Then $A$ admits a $\mathbb{Z}$-field of fractions.
\end{theorem}

\begin{remark} \label{rem:quadricsfunctionfields}
As three dimensional, quadratic AS-regular $\mathbb{Z}$-algebras are 1-periodic, the existence of their $\mathbb{Z}$-field of fractions is automatic from Proposition \ref{prp:gradedZfractions}.
\end{remark}

\begin{remark}
When $Q=\Frac(A)$ it is customary to refer to $Q_{0,0}$ as the function field of $\QGr(A)$. We will do so throughout this paper.
\end{remark}

\section{Summary of the results in \cite{PresVdB}}
\label{sec:PresVdB}

In \cite{PresVdB} we construct noncommutative versions of the birational transformation $\mathbb{P}^1 \times \mathbb{P}^1 \dashrightarrow \mathbb{P}^2$ and the Cremona Transform $\mathbb{P}^2 \dashrightarrow \mathbb{P}^2$ using the following recipe:

\begin{enumerate}
\item[Step 1)] Let $A$ be a 3-dimensional Sklyanin algebra. Let $(Y,\Lscr,\sigma)$ be the associated geometric data, $B$ be the associated twisted homogeneous coordinate ring and $\Qcoh(X) = \QGr(A)$. Denote $\Bimod(Y-Y) \hookrightarrow \Bimod(X-X)$ for the categories of bimodules (recall: $\Bimod(X_1-X_2)$ is the category of right exact functors $\Qcoh(X_1) \rightarrow \Qcoh(X_2)$ commuting with direct limits) with $o_Y, o_X$ corresponding to the identity functors. Let $o_X(-Y) = \ker \left( o_X \twoheadrightarrow o_Y \right)$\footnote{Unfortunately $\Bimod(X_1-X_2)$ appears not to be an abelian category.% and hence kernels and cokernels need not always exist. 
This technical difficulty is solved in \cite{VdB19} by embedding $\Bimod(X-Y)$ into a larger abelian category $\BIMOD(X_1-X_2)$ consisting of ``weak bimodules''. 
%The category $\BIMOD(X_1-X_2)$ is opposite to the category of left exact functors $\Qcoh(X_1)\rightarrow \Qcoh(X_2)$. Since left exact functors are determined by their values on injectives, they trivially form an abelian category. The category $\Bimod(X_1-X_2)$ is the full category of $\BIMOD(X_1-X_2)$ consisting of functors commuting with direct products. 
In particular we will always construct kernels and cokernels as elements in $\BIMOD(X_1-X_2)$. This being said, these technical complication will be invisible in this paper as all bimodules we construct will be elements of the (smaller) category $\Bimod(X_1-X_2)$. }.

%\item[Step 1)] Define a global sections functor $\Gamma(X,-): \Bimod(X-X) \rightarrow \Vect_k: \Nscr \mapsto \Hom(\Oscr_X , \Oscr_X \otimes \Nscr)$ where $\Oscr_X = \pi(A)$. Let $o_X(i)$ be the shift functors on $X$ and check that $\Gamma(X,o_X(i)) = A_i$. Similarly one defines the global sections of bimodules on $Y$.
\item[Step 2)] Choose a divisor $d$ on $Y$ in the following way:
\begin{itemize}
\item $d= p + q + r$ for $p,q,r$ distinct non-collinear points in case $A$ is quadratic
\item $d = p$ in case $A$ is cubic
\end{itemize}
Let $\Oscr_d$ be it's structure sheaf and denote $o_d \in \Bimod(Y-Y) \subset \Bimod(X-X)$ for the bimodule
\[ o_d := ( - \otimes_Y \Oscr_d) \]

%If $\Nscr\in \Qcoh(Y)$ then we may consider it as an object in $\Bimod(Y-Y)$. 

and 
\begin{align}\label{eq:defmd}
 m_d & := \ker(o_X \twoheadrightarrow o_d) \\
\notag m_{d,Y} & := \ker( o_Y \twoheadrightarrow o_d)
\end{align}
%\item[Step 2)]
%Define $\mathbb{Z}$-algebra objects $\Dscr_Y, \Dscr$ in $\Bimod(Y-Y)$ and $\Bimod(X-X)$ respectively via
%\begin{align*} 
%(\Dscr_Y)_{m,n}& = 
%\begin{cases}
%o_Y(-2m) \otimes_{o_Y} m_{\tau^{-m} d,Y} \ldots m_{\tau^{-n+1} d,Y} \otimes_{o_Y} o_Y(2n) & \text{if $n\ge m$}\\
%0 & \text{if $n< m$}
%\end{cases}\\
%\Dscr_{m,n} &=
%\begin{cases}
% o_X(-2m) \otimes_{o_X} m_{\tau^{-m} d} \ldots m_{\tau^{-n+1} d} \otimes_{o_X} o_X(2n) &\text{if $n\ge m$}\\
%0&\text{if $n<m$}
%\end{cases}
%\end{align*}
\item[Step 3)] Define $\mathbb{Z}$-algebras $D_Y$ and $D$ as follows
% denote the global sections of $\Dscr_Y$ and $\Dscr$ respectively. I.e. $D_Y$ and $D$ are $\mathbb{Z}$-algebras with $(D_Y)_{m,n} = \Gamma\left( Y,(\Dscr_Y)_{m,n} \right)$ and $D_{m,n} = \Gamma\left( X,\Dscr_{m,n} \right)$. 
\begin{align} 
\notag (D_Y)_{m,n}& = 
\begin{cases}
\Hom \left( \Oscr_Y(-2n), \Oscr_Y(-2m) \otimes_{o_Y} m_{\tau^{-m} d,Y} \ldots m_{\tau^{-n+1} d,Y} \right) & \text{if $n\ge m$}\\
0 & \text{if $n< m$}
\end{cases}\\
\label{eq:definitionDandDY} & \\
\notag D_{m,n} &=
\begin{cases}
\Hom \left( \Oscr_X(-2n), \Oscr_X(-2m) \otimes_{o_X} m_{\tau^{-m} d} \ldots m_{\tau^{-n+1} d} \right) &\text{if $n\ge m$}\\
0&\text{if $n<m$}
\end{cases}
\end{align}
where $\tau = \sigma^{s+1}$,
\begin{eqnarray} \label{eq:OX1}
\Oscr_X(i) = \pi(A(i)), \ \Oscr_Y(i) = \pi(B(i))
\end{eqnarray} and $\pi$ is the quotient functor $\Gr(A) \rightarrow \QGr(A)$.\\
By \cite[Lemma 8.2.1]{VdB19} the inclusions $m_{\tau^{-i}d} \hookrightarrow o_X$ give rise to an inclusion $D \hookrightarrow \check{A}^{(2)}$.
\item[Step 4)] One shows that $D_Y$ is a twisted homogeneous coordinate ring, i.e. there is a sequence of line bundles $(\Gscr_i)_{i \in \mathbb{Z}}$ such that
\begin{eqnarray}
\label{eq:firstGi} (D_Y)_{m,n} = \Gamma(Y, \Gscr_m \otimes \ldots \otimes \Gscr_{n-1}) 
\end{eqnarray}
Moreover the $\Gscr_i$ must form an ``elliptic helix'' (for a quadratic AS regular $\mathbb{Z}$-algebra), i.e.
\begin{itemize}
\item $\deg \Gscr_i = 3$ for all $i$
\item $\Gscr_0 \not \cong \Gscr_1$
\item $\Gscr_i \otimes \Gscr_{i+1}^{\otimes -2} \otimes \Gscr_{i+2} \equiv \Oscr_Y$
\end{itemize}
\item[Step 5)]
One shows that $D$ is generated in degree 1. For this we need some natural vanishing results on $\Oscr_X(a) \otimes m_{\tau^{-m} d} \ldots m_{\tau^{-n+1} d}$ %based on a resolution for $m_d$ and the standard vanishing on $A$. 
(We will recapitulate and slightly extend these vanishing results in Lemmas \ref{cor:standardvanishing2} and \ref{cor:standardvanishing}.)
\item[Step 6)]
One shows that the canonical map $D \rightarrow D_Y$ is surjective and that $D$ has the correct Hilbert series. It then follows that $D$ is the AS-regular $\mathbb{Z}$-algebra corresponding to the elliptic helix $(\Gscr_i)_{i \in \mathbb{Z}}$ on $Y$. As $D$ is quadratic, there is an AS-regular quadratic algebra $A'$ such that $D = \widecheck{A'}$ (see \cite[Theorem 4.2.2]{VdB38}).
\item[Step 7)]
Check that the induced morphism 
\[ \Frac_0(A') = \Frac_{0,0}(D) \rightarrow \Frac_{0,0}(\check{A}) = \Frac_0(A) \]
 is an isomorphism. This check is immediate as injectivity comes for free and surjectivity follows by a Hilbert series argument.
\end{enumerate}

\section{Addendum to \cite{PresVdB}}
\label{sec:addendum}
%In this section we extend some of the results in \cite{PresVdB}. As most of the proofs can be taken from loc.cit. we only prove the essentially new steps and refer to loc.cit. for all other proofs. 
As explained in the previous section, in \cite{PresVdB} we construct noncommutative versions of the birational transformation $\mathbb{P}^1 \times \mathbb{P}^1 \dashrightarrow \mathbb{P}^2$. In this section we slightly generalize this (Theorem \ref{thm:firstextension}) and in addition we discuss a noncommutative version of the inverse birational transformation $\mathbb{P}^2 \dashrightarrow \mathbb{P}^1 \times \mathbb{P}^1$ (Theorem \ref{thm:secondextension}). %In particular we prove the following:

\begin{theorem} \label{thm:firstextension}
Let $A = A(Y,(\Lscr_i)_{i\in \mathbb{Z}})$ be a quadric such that
\begin{itemize}
\item $Y$ is a smooth elliptic curve
\item $\Lscr_2 \cong \alpha^* \Lscr_0$ for some $\alpha \in \Aut(Y)$ which is given by translation by a point of order at least 3 (i.e. $\alpha^2 \neq Id$).
\end{itemize}
then there exists a quadratic Sklanin algebra $A'$ and an inclusion $\widecheck{A'} \hookrightarrow A^{(2)}$ inducing an isomorphism of function fields.
\end{theorem}

\begin{remark} \label{rem:defsklyaninquadric}
%From \cite[Proposition 5.6.1]{VdB38} we know that for any quadric (i.e. any cubic AS-regular $\mathbb{Z}$-algebra) there exists an $\alpha \in \Aut(Y)$ such that $\Lscr_2 \cong \alpha^* \Lscr_0$ (and hence by the elliptic helix condition: $\Lscr_{i+2} \cong \alpha^* \Lscr_i$ for all $i$). We require $\alpha$ to be given by translation by a point of order at least 3 and will refer to such a quadric as a cubic Sklyanin $\mathbb{Z}$-algebra.
We refer to a quadric as in Theorem \ref{thm:firstextension} as a cubic Sklyanin $\mathbb{Z}$-algebra.
\end{remark}
\begin{remark}
Theorem \ref{thm:firstextension} extends \cite{PresVdB} in the sense that it does not require $A$ to be 1-periodic.%hence allowing $A$ to be a so called quadric: a cubic AS-regular $\mathbb{Z}$-algebra which does not come from a graded AS-regular algebra. %This is necessary because the maps in Theorem \ref{thm:firstextension} will be inverse to the ones in Theorem \ref{thm:secondextension} which can give rise to cubic AS-regular algebras which are not 1-periodic.
\end{remark}
\begin{theorem} \label{thm:secondextension}
Let $A = A(Y,\psi,\Lscr)$ be a quadratic Sklyanin algebra. Then there exists a cubic Sklyanin $\mathbb{Z}$-algebra $D$ and an inclusion $D \hookrightarrow \check{A}$ inducing an isomorphism of function fields. The construction of $D$ depends on two points, $p,q \in Y$ and $D$ is 1-periodic, i.e. of the form $D = \widecheck{A'}$, if $q = \sigma p$ for $\sigma$ a translation such that $\sigma^2=\psi^{-3}$.
%these points satisfy $q = \psi^{-3/2} p$.
\end{theorem}

As both Theorem \ref{thm:firstextension} and \ref{thm:secondextension} contain statements on AS-regular $\mathbb{Z}$-algebras which need not be 1-periodic we introduce some notation for an AS-regular $\mathbb{Z}$-algebra $A$:
%in this case we no longer have the existence of a shift functor $o_X(1)$ and 
Inspired by $(\ref{eq:OX1})$ and Remark \ref{rem:OX} we define 
\begin{eqnarray} \label{eq:OX2}
\Oscr_X(i) = \pi( e_{-i}A), \ \Oscr_Y(i)= \pi( e_{-i}B)
\end{eqnarray}
Under the equivalence of categories $\QGr(B) \cong \Qcoh(Y)$ (see \eqref{eq:YisB}) $\Oscr_Y(i)$ can be identified with
 \begin{eqnarray} \label{eq:identifyOY} \begin{cases} \Lscr_0 \Lscr_1 \ldots \Lscr_{i-1} & \textrm{ if } i > 0 \\ \Oscr_Y & \textrm{ if } i=0 \\ (\Lscr_0 \Lscr_1 \ldots \Lscr_{-i-1})^{-1} & \textrm{ if } i<0 \end{cases} \end{eqnarray}
 
Similar to the graded case \cite[\S 4]{PresVdB} we have
\begin{eqnarray*}
\Hom ( \Oscr_X(-j), \Oscr_X(-i)) & = & \Hom_{\QGr(A)}( \pi(e_jA), \pi(e_iA)) \\
 & = & \lim_{n \rightarrow \infty} \Hom_{\Gr(A)}( e_jA_{\geq n}, e_iA) \\
 & = & \Hom_{\Gr(A)}( e_jA, e_iA) \\ 
 & = & A_{i,j}
\end{eqnarray*}
where the second equality follows from the fact that $e_jA$ and $e_iA$ are finitely generated $A$-modules (see \cite[Proposition 7.2]{artinzhang} for the graded case). The third equality follows from the Gorenstein condition as this implies $\Hom_{\Gr(A)}( e_jA/e_jA_{\geq n}, e_iA) = \Ext^1_{\Gr(A)}( e_jA/e_jA_{\geq n}, e_iA)=0$.

We also have the following standard vanishing result

\begin{lemma}\label{lem:standardvanishing}
Let $A$ be an AS-regular $\mathbb{Z}$-algebra of dimension 3 with Hilbert function $h$ and let $X = \QGr(A)$ and $\Oscr_X(i)$ be as above, then we have
\begin{eqnarray*}
\Ext^n(\Oscr_X(i), \Oscr_X(j)) & = & 0 \textrm{ for $n \neq 0,2$} \\ 
\dim_k \left(\Ext^2(\Oscr_X(i), \Oscr_X(j)) \right) & = & \begin{cases} 0 & \textrm{ if $j \geq i-s$} \\ h(i-s-j) & \textrm{ if $j \leq i-s-1$} \end{cases}
\end{eqnarray*}
\end{lemma}
\begin{proof}
Similar to \cite[Theorem 8.1]{artinzhang}.
\end{proof}
\begin{remark}
Throughout the rest of the paper $X$ and $\Oscr_X(i)$ will be defined as in Lemma \ref{lem:standardvanishing}.
\end{remark}
We now give an overview of both proofs separately. The recipe is as in \S \ref{sec:PresVdB} and we only highlight the steps that need to be adapted.
\subsection{Noncommutative $\mathbb{P}^1 \times \mathbb{P}^1 \dashrightarrow \mathbb{P}^2$. (Proof of Theorem \ref{thm:firstextension})}
\label{subsec:quadric}
\begin{enumerate}
\item[Step 1)]
Let $A = A(Y,(\Lscr_i)_{i \in \mathbb{Z}})$ be a cubic Sklyanin $\mathbb{Z}$-algebra as in the statement of Theorem \ref{thm:firstextension}.
\item[Steps 2) and 3)]
Let $p$ be a point on $Y$ and set $\tau = \alpha^2 \neq Id$ and $d=p$ in the definition of $D$ and $D_Y$ as in \eqref{eq:definitionDandDY}.

%We let $D_Y, D$ be the $\mathbb{Z}$-algebras given by the global sections of $\Dscr_Y$, respectively $\Dscr$, i.e.
%\begin{eqnarray*} 
%D_{Y,m,n} & = & \begin{cases} \Hom \left(\Oscr(Y), \Oscr_Y(-2m) \otimes_{o_Y} m_{\tau^{-m} p,Y} \ldots m_{\tau^{-n+1} p,Y} \otimes_{o_Y} o_Y(2n) \right) &  \text{if $n\ge m$}\\
%0 & \text{if $n< m$} \end{cases} \\
%D_{m,n} & = & \begin{cases} \Hom \left(\Oscr(X), \Oscr_X(-2m) \otimes_{o_X} m_{\tau^{-m} p} \ldots m_{\tau^{-n+1} p} \otimes_{o_X} o_X(2n) \right) &  \text{if $n\ge m$}\\
%0 & \text{if $n< m$} \end{cases} \end{eqnarray*}

%The inclusion $\Dscr_{m,n} \hookrightarrow o_X(2(n-m))$ gives rise to
%an inclusion of $\mathbb{Z}$-algebras $D \hookrightarrow A^{(2)}$ by using 
%\cite[Lemma 8.2.1]{VdB19} with $\Escr=\Oscr_X$.\\

\item[Step 4)]
Computations similar to the ones in \cite[\S 5]{PresVdB} show that $D_Y$ is a twisted homogeneous coordinate ring, i.e.
\[ D_{Y,m,n} = \Gamma \left( Y,\Gscr_{m} \otimes \ldots \otimes \Gscr_{n-1} \right) \]
for some new collection of line bundles on $Y$:
\begin{eqnarray}
\label{eq:Gscriquadric} \Gscr_i = \Oscr_Y(- \tau^{-i} p) \otimes \Lscr_{2i} \otimes \Lscr_{2i+1} \end{eqnarray}
It follows easily that these line bundles satisfy
\begin{enumerate}
\item $\deg \Gscr_i =3$
\item $\Gscr_0\not\cong\Gscr_1$.
\item $\Gscr_i \otimes \Gscr_{i+1}^{\otimes -2}\otimes \Gscr_{i+2} \cong\Oscr_Y$.
\item $\psi^{\ast}(\Gscr_i) \cong \Gscr_{i+1}$
 where $\psi$ is an arbitrary translation satisfying $\psi^3=\tau$.
\end{enumerate}

%NOTE: I THINK THAT FOR THE SECOND ITEM IN THIS LIST WE ONLY NEED $\psi^3=\alpha^2 \neq Id$?\\

\item[Steps 5), 6) and 7)]

We need the following generalizations of \cite[Lemma 6.2, Lemma 6.4 and Lemma 6.5]{PresVdB}:
\begin{lemma} \label{lem:quadricresolution}
Let $A$ be as above and let $p$ be a point in $Y$. Let $i \in \mathbb{Z}$ and let $P_i$ be the point module generated in degree $i$ corresponding to the point $p$ (i.e. the point module truncated at degree $i$). 

Then there is a complex of the following form:
\begin{equation}
\label{minimal}
0\rightarrow e_{i+5}A \xrightarrow{(\zeta,0)} e_{i+4}A^{\oplus 2} \oplus e_{i+3}A \rightarrow e_{i+2}A^{\oplus 3} \rightarrow e_iA\rightarrow P_i\rightarrow 0
\end{equation}
where $\zeta$ is part of the minimal resolution of $S_{i+1}$ 
 as given in \cite[Definition 4.1.1.]{VdB19} 
\[
0\rightarrow e_{i+5}A \xrightarrow{\zeta} e_{i+4}A^{\oplus 2}\xrightarrow{\varepsilon} e_{i+2}A^{\oplus 2}\xrightarrow{\delta_0} e_{i+1}A\xrightarrow{\gamma} S_{i+1}\rightarrow0
\]
Moreover the complex \eqref{minimal} is exact everywhere except at $e_i A$ where it has one-dimensional cohomology, concentrated in degree $i+1$. 
\end{lemma}
\begin{proof}
Similar to the proof of \cite[Lemma 6.2]{PresVdB} this is based on fact that the minimal projective resolution for $P_i$ is given by
\begin{eqnarray}
\label{eq:minimalprojectiveq}
0\rightarrow e_{i+3}A\rightarrow e_{i+1}A \oplus e_{i+2}A \rightarrow e_iA\rightarrow P_i \rightarrow 0
\end{eqnarray}
The latter is a $\mathbb{Z}$-algebra analogue of the projective resolution of a point module as in \cite[Proposition 6.7.i]{ATV2}. The proof in loc. cit. can be adapted to a $\mathbb{Z}$-algebra version as follows:\\
First we note that by \cite[\S 5]{VdB19} there is a 1-1-correspondence between point modules (truncated in degree $i$) and points on $Y$, so the correspondence between $p \in Y$ and $P_i$ is well defined.\\
Next we generalize line modules to modules of the form $e_j A/ a A$ where $a \in A_{j,j+1}$. As $A$ is a $\mathbb{Z}$-domain (Theorem \ref{thm:quadricdomain}) we know that line modules have the desired Hilbert series and have projective dimension 1. The characterization of line modules as in \cite[Corollary 2.43]{ATV2} generalizes to this new definition of line modules if we replace the modules $A(-i)$ in the graded case by $e_i A$.\\
Finally we find $(\ref{eq:minimalprojectiveq})$ by generalizing the proof of \cite[Proposition 6.7.i]{ATV2}. This proof uses the fact that point modules in the graded case have projective dimension 2. A careful observation however shows that projective dimension $\leq 2$ suffices for the proof. This in turn can be shown by a variation of \cite[Proposition 2.46.i]{ATV2}.
\end{proof}

In particular \cite[Lemma 6.4, Lemma 6.5]{PresVdB} generalize to the following

\begin{lemma} \label{cor:standardvanishing2}
With the notations as above we have for $i-j \leq 3$:
\[ \Ext^2( \Oscr_X(i), \Oscr_X(j) \otimes m_{\tau^{-m}d} \ldots m_{\tau^{-n+1}d}) = 0 \]
and for $i-j \leq 2m-2n+2$:
\[ \Ext^1( \Oscr_X(i), \Oscr_X(j) \otimes m_{\tau^{-m}d} \ldots m_{\tau^{-n+1}d}) = 0 \]
\end{lemma}
\begin{proof}
Similar to \cite[Lemma 6.4, Lemma 6.5]{PresVdB} by using Lemmas \ref{lem:quadricresolution} and \ref{lem:standardvanishing}.
\end{proof}

One can now copy the proofs in \cite[\S6]{PresVdB} to find that $D$ and $D_Y$ are generated in degree 1, that the canonical map $D \rightarrow D_Y$ is surjective and an isomorphism in degree 0,1 and 2, that $D$ has the correct Hilbert series and hence $D \cong A(Y,(\Gscr_i)_{i \in \mathbb{Z}}) = \widecheck{A'}$ for the graded algebra $A' = A'(Y,\Gscr_0, \psi)$. Similarly the results in \cite[\S7]{PresVdB} extend to an isomorphism of function fields $\Frac_{0,0}(A) \cong \Frac_0(A')$.

%NOTE: still include: why do we have existence of the function field?

\end{enumerate}
\subsection{Noncommutative $\mathbb{P}^2 \dashrightarrow \mathbb{P}^1 \times \mathbb{P}^1$ (Proof of Theorem \ref{thm:secondextension})}
\label{subsec:quadricinverse}

\begin{enumerate}
\item[Step 1)]
Throughout this section $A = A(Y,\Lscr,\psi)$ is a quadratic Sklyanin algebra, hence $Y$ is a smooth elliptic curve, $\Lscr$ is a degree 3 line bundle on $Y$ and $\psi \in \Aut(Y)$ is given by translation by a point with $\tau = \psi^3 \neq Id$. %As $A$ is a graded algebra, all shift functors exist and as usual they satisfy $\Gamma(X, o_X(i)) = A_i$

\item[Steps 2) and 3)]
The construction of $D$ is slightly different from the one in \eqref{eq:definitionDandDY}:\\
Let $p,q$ be points on $Y$ and denote
\begin{eqnarray} \label{definitiondi}
d_i = \begin{cases} 
\tau^{-j}p & \text{if $i=2j$} \\
\tau^{-j}q & \text{if $i=2j+1$} \\
\end{cases}
\end{eqnarray}
where $\tau = \psi^3$. We define $m_{d_i,Y}$ and $m_{d_i}$ as in $(\ref{eq:defmd})$ and introduce $\mathbb{Z}$-algebra $D_Y$ and $D$ via
\begin{align} 
(D_Y)_{m,n}& = 
\begin{cases}
\Hom(\Oscr_Y(-n), \Oscr_Y(-m) \otimes_{o_Y} m_{d_{m},Y} \ldots m_{d_{n-1},Y} ) & \text{if $n\geq m$}\\
0 & \text{if $n< m$}
\end{cases}\\
D_{m,n} &=
\begin{cases}
\Hom(\Oscr_X(-n), \Oscr_X(-m) \otimes_{o_X} m_{d_{m}} \ldots m_{d_{n-1}} ) &\text{if $n\geq m$}\\
0&\text{if $n<m$}
\end{cases}
\end{align}
Again using \cite[Lemma 8.2.1]{VdB19}, there is an inclusion $D \hookrightarrow \check{A}$.
\item[Step 4)]

Computations similar to the ones in \cite[\S 5]{PresVdB} show
\[ D_{Y,m,n} = \Gamma \left( Y,\Gscr_m \otimes \ldots \otimes \Gscr_{n-1} \right) \]
for some collection of line bundles on $Y$ given by:
\begin{eqnarray} \label{eq:Gscriquadratic} \Gscr_{i} = \Oscr_Y(-d_i) \otimes \psi^{i*} \Lscr  \end{eqnarray}
and these line bundles satisfy
\begin{enumerate}
\item $\deg \Gscr_{i}=2$
\item $\Gscr_{0}\not\cong\Gscr_{2}$.
\item $\Gscr_{i}\otimes \Gscr_{i+1}^{-1}\otimes \Gscr_{i+2}^{-1} \otimes \Gscr_{i+3} \cong\Oscr_Y$.
\item $\alpha^{\ast}(\Gscr_i) \cong \Gscr_{i+2}$
 where $\alpha$ is an arbitrary translation satisfying $\alpha^2=\tau$.
\end{enumerate}
Moreover if $q = \alpha^{-1} p$ then $\Gscr_{i+1} \cong \sigma^* \Gscr_i$ for an arbitrary translation ${\sigma: Y \rightarrow Y}$ such that $\sigma^2 = \alpha$.

\item[Steps 5) and 6)]
This is completely analogous to \cite[\S6]{PresVdB}. The only difference lies in the resolution of $\Oscr(-i) \otimes m_{d_n}$: this time it is of the form:
\[ 0 \rightarrow \Oscr(-i-3) \rightarrow \Oscr(-i-2)^{\oplus 2} \rightarrow \Oscr(-i) \otimes m_{d_n} \rightarrow 0 \] 
This is based on the following:
\begin{lemma} \label{lem:quadraticresolution}
Let $A=A(Y,\Lscr,\psi)$ be a quadratic AS-regular algebra of dimension 3, let $p \in Y$ be a point and let $P$ be the corresponding point module. Then the minimal resolution of $P$ has the following form
\[ 0 \rightarrow A(-2) \rightarrow A(-1)^{\oplus 2} \rightarrow A \rightarrow P \rightarrow 0 \]
\end{lemma}
\begin{proof}
This follows immediately from \cite[Proposition 6.7]{ATV2}.
\end{proof}
From this we find the following standard vanishing results:
\begin{lemma} \label{cor:standardvanishing}
With the notations as above we have for $i-j \leq 2$:
\[ \Ext^2( \Oscr_X(i), \Oscr_X(j) \otimes m_{m_d} \ldots m_{d_{n-1}}) = 0 \]
and for $i-j \leq m-n+1$:
\[ \Ext^1( \Oscr_X(i), \Oscr_X(j) \otimes m_{d_m} \ldots m_{d_{n-1}}) = 0 \]
\end{lemma}
\begin{proof}
Similar to \cite[Lemma 6.4, Lemma 6.5]{PresVdB} by using Lemmas \ref{lem:quadraticresolution} and \ref{lem:standardvanishing}.
\end{proof}
For checking that $D$ has the correct Hilbert series we see that (using computation similar to \cite[Corollary 5.2.4]{VdB19}) the colength of $m_{d_{m}} \ldots m_{d_{m+i-1}}$ inside $o_X$ is
\begin{eqnarray} \label{eq:colength} \begin{cases}  \displaystyle
2 \cdot \frac{a(a+1)}{2} =  a(a+1) & \textrm{ if $i=2a$}\\
\displaystyle \frac{(a+2)(a+1)}{2} + \frac{a(a+1)}{2} = (a+1)^2 & \textrm{ if $i=2a+1$}
\end{cases}
\end{eqnarray}
Using Lemma \ref{cor:standardvanishing} this implies
\[ \dim D_{m,m+i} = \begin{cases}
\displaystyle \frac{(2a+1)(2a+2)}{2} - a(a+1) = (a+1)^2 & \textrm{ if $i=2a$}\\
\displaystyle \frac{(2a+2)(2a+3)}{2} -(a+1)^2 = (a+1)(a+2) & \textrm{ if $i=2a+1$}
\end{cases}
\]
i.e. $D$ has the Hilbert series of a cubic AS-regular $\mathbb{Z}-algebra$ and hence we can conclude that $D$ is AS-regular.

\item[Step 7)]
Injectivity of $\Frac_{0,0}(D) \rightarrow \Frac_0(A)$ is again immediate. For surjectivity we need to prove that for any fixed $n \in \mathbb{N}$ there is some $N \in \mathbb{N} $ such that $\Hom(\Oscr_X(-i-n-N),\Oscr_X(-i-n) \otimes m_{d_{i}} \ldots m_{d_{i+n+N-1}}) \neq 0$ holds for all $i \in \mathbb{Z}$. Now note that the colength of $m_{d_{i}} \ldots m_{d_{i+n+N-1}}$ inside $o_X$ is given by $a^2$ or $a(a+1)$ if $n+N = 2a$ or $2a+1$ respectively. In particular this colenth grows like $\frac{N^2}{4}$ for $N$ large. Hence the codimension of $\Hom(\Oscr_X(-i-n-N),\Oscr_X(-i-n) \otimes m_{d_{i}} \ldots m_{d_{i+n+N-1}} )$ inside $A_{i+n,i+n+N}$ grows at most like $\frac{N^2}{4}$. %as we might no longer have H^1-vanishing for large N we only get the upperbound but this suffices
As $\dim_k A_{i+n,i+n+N}$ grows like $\frac{N^2}{2}$ it hence follows that $\Gamma(\Oscr_X(-i-n-N),\Oscr_X(-i-n) \otimes m_{d_{i}} \ldots m_{d_{i+n+N-1}}) \neq 0$ for large $N$.
\end{enumerate}

\newpage
\section{Cremona Transformations and quadratic transforms}
\begin{convention}
In the following sections we will always identify graded algebras with their associated $\mathbb{Z}$-algebras. The $\mathbb{Z}$-algebra $D$ which was introduced in sections \ref{sec:PresVdB} and \ref{sec:addendum} will now be denoted by $A'$. Moreover we will use the terminology ``(three dimensional) Sklyanin $\mathbb{Z}$-algebras'' for both quadratic and cubic Sklyanin $\mathbb{Z}$-algebras.
\end{convention}
\label{sec:cremona}
The goal of this section is to link the noncommutative version of the Cremona transform $\mathbb{P}^2 \dashrightarrow \mathbb{P}^2$ as constructed in \cite{PresVdB} (see also \S \ref{sec:PresVdB} for a reminder of this construction) to the noncommutative versions of $\mathbb{P}^1 \times \mathbb{P}^1 \dashrightarrow \mathbb{P}^2$ and $\mathbb{P}^2 \dashrightarrow \mathbb{P}^1 \times \mathbb{P}^1$ as in Theorems \ref{thm:firstextension}, \ref{thm:secondextension}. This is inspired by the following commutative picture:

\begin{figure}[ht]
\includegraphics[width=15cm]{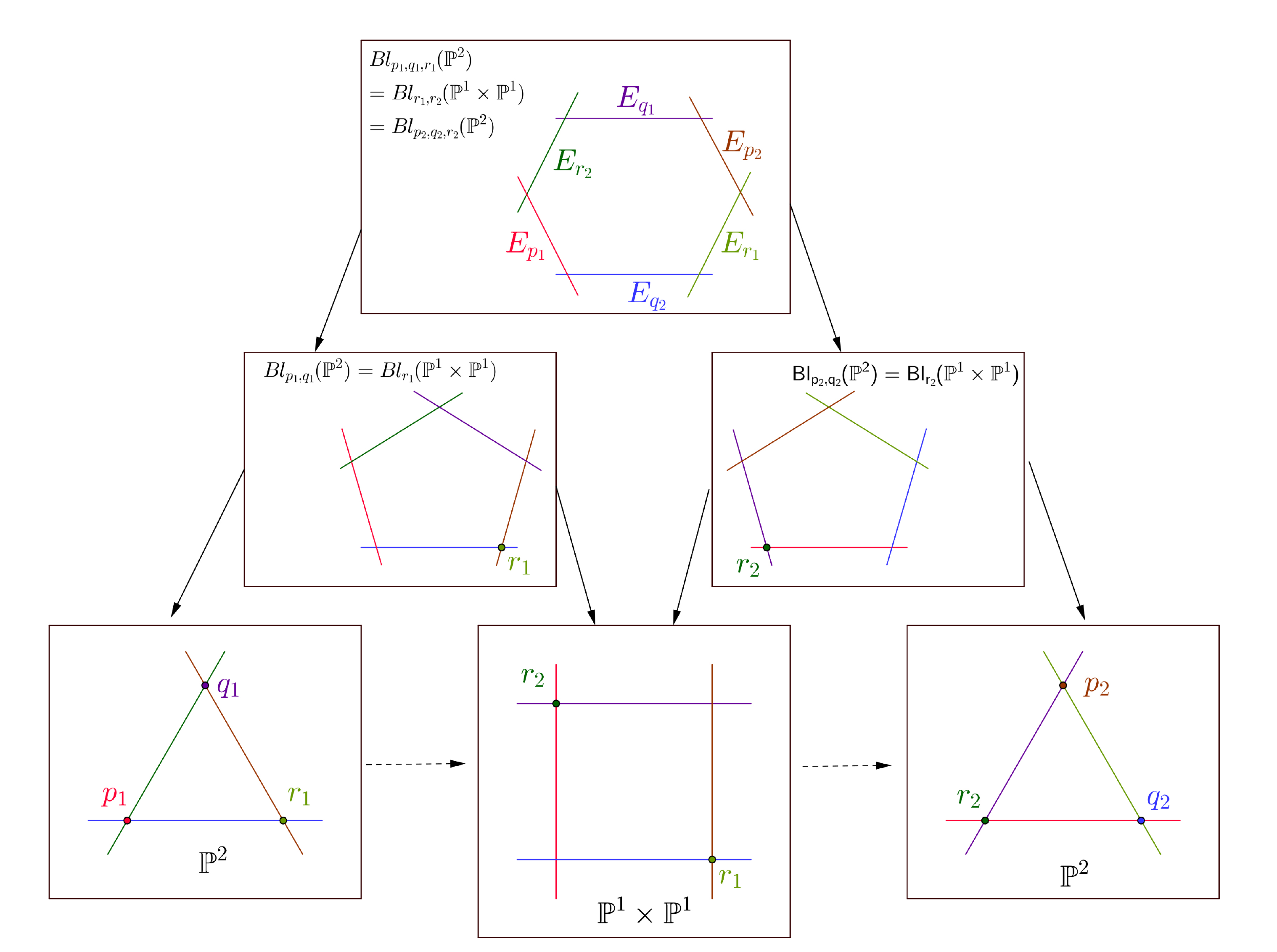}
\caption{The commutative Cremona transform}
\label{fig:cremona}
\end{figure}

i.e. the (commutative) Cremona transform $\gamma: \mathbb{P}^2 \dashrightarrow \mathbb{P}^2$ (given by blowing up three non-collinear points $p_1, q_1, r_1$, see for example \cite[Example V.4.2.3.]{hartshorne}) factors as $\gamma = \gamma_1 \circ \gamma_2$ where $\gamma_1: \mathbb{P}^1 \times \mathbb{P}^1 \dashrightarrow \mathbb{P}^2$ and $\gamma_2: \mathbb{P}^2 \dashrightarrow \mathbb{P}^1 \times \mathbb{P}^1$ are the standard birational transformations obtained by blowing up one point in $\mathbb{P}^1 \times \mathbb{P}^1$ and two points in $\mathbb{P}^2$.\\

We prove that (modulo some technical assumptions) the same holds in the noncommutative setting:
\begin{theorem} \label{thm:cremona}
Let $A=A(Y,\Lscr,\psi)$ and $A'$ be quadratic Sklyanin algebras and let $\gamma: A' \hookrightarrow A^{(2)}$ be a noncommutative Cremona transformation as in \cite{PresVdB}. Denote $d=p+q+r$ as in loc. cit. and assume $p,q$ and $r$ are non-collinear points lying in different $\tau$-orbits where $\tau= \psi^3$. Then there exists a cubic Sklyanin $\mathbb{Z}$-algebra $A''$ and inclusions $\gamma_1: A' \hookrightarrow A^{\prime \prime (2)}$, $\gamma_2: A'' \hookrightarrow A$ such that $\gamma_1$ is as in Theorem \ref{thm:firstextension}, $\gamma_2$ is as in Theorem \ref{thm:secondextension} and $\gamma = \gamma_2 \circ \gamma_1$.
\end{theorem}

Before we can prove this we need the following technical result:

\begin{lemma} \label{lem:msplits}
Let $p$ and $q$ be points such that $p \neq \tau^i q$ for $i \in \{ -1, 0 , 1 \}$. Then $m_{p+q} = m_p m_q$.
\end{lemma}
\begin{proof}
Consider the following diagram:

\begin{center}
\begin{tikzpicture}
\matrix(m)[matrix of math nodes,
row sep=4em, column sep=4em,
text height=1.5ex, text depth=0.25ex]
{m_p \otimes o_q & o_q & \\
m_p & o_X & o_p \\
m_p \otimes m_q & m_q & o_p \otimes m_q \\};
\path[->,font=\scriptsize]
(m-1-1) edge node[above]{$a_1$} (m-1-2)
(m-2-1) edge node[above]{$a_2$}(m-2-2)
        edge node[left]{$g_1$}(m-1-1)
(m-3-1) edge node[below]{$a_3$}(m-3-2)
        edge node[left]{$f_1$}(m-2-1)
(m-2-2) edge node[above]{$b_2$}(m-2-3)
        edge node[left]{$g_2$}(m-1-2)
(m-3-2) edge node[below]{$b_3$}(m-3-3)
        edge node[right]{$f_2$}(m-2-2) 
(m-3-3) edge node[right]{$f_3$}(m-2-3)
;
\end{tikzpicture}
\end{center}
The middle row and column are the exact sequences as in \eqref{eq:defmd}. The last row is obtained by applying $- \otimes m_q$ to the middle row, it hence is automatically right exact. %It is also left-exact as $\Tor_1(o_p,m_q)=0$ (which can be checked locally by completion as in \cite[(5.35)]{VdB19}?). 
Similarly the first column is right exact.\\

We can identify $m_pm_q$ both with the image of $f_1$ and $a_3$ (and hence the kernel of $g_1$ and $b_3$). To see this recall that $m_pm_q$ is defined as the image of 
\[ f: m_p \otimes m_q \rightarrow o_X \otimes o_X \rightarrow o_X \]
The identification then follows as $f= f_2 \circ a_3 = a_2 \circ f_1$ and as $a_2$ and $f_2$ are monomorphisms.\\

Next we claim that $a_1$ and $f_3$ are in fact isomorphisms. The proof follows from this claim as it implies  $m_p m_q$ serves both as a kernel for $b_2 \circ f_2$ and $g_2 \circ a_2$. As such it is the pullback of
\begin{center}
\begin{tikzpicture}
\matrix(m)[matrix of math nodes,
row sep=4em, column sep=4em,
text height=1.5ex, text depth=0.25ex]
{m_p & o_X \\
 & m_q \\};
\path[->,font=\scriptsize]
(m-1-1) edge node[above]{$a_2$} (m-1-2)
(m-2-2) edge node[right]{$f_2$}(m-1-2)
;
\end{tikzpicture}
\end{center}
On the other hand $o_{p+q} = o_p \oplus o_q$ such that $m_{p+q}$, being the kernel of ${o_X \xrightarrow{b_2 \oplus g_2} o_{p+q}}$, also is a pullback of the above diagram.
In particular $m_{p+q} \cong m_p m_q$. \\

It hence remains to prove the above claim. %Again this can be checked locally.\\
As the argument is the same for both morphisms we only explain this for $a_1$. Note that this map is obtained by tensoring $m_p \hookrightarrow o_X$ with $o_q$. As such the cokernel of $a_1$ is given by $o_p \otimes o_q$ whereas the kernel is a subobject of $\mathcal{T}or_1(o_p,o_q)$ (see \cite[\S3]{VdB19} for the definition of $\mathcal{T}or_i$). In particular it suffices to show $o_p \otimes o_q = 0 = \mathcal{T}or_1(o_p,o_q)$. By \cite[Lemma 5.5.1]{VdB19} we need to show $\Hom_X(\Oscr_p,\Oscr_q)= 0 =\Ext^1_X(\Oscr_p,\Oscr_q)$. For $\Hom_X$ this follows from the fact that $\Oscr_p$ and $\Oscr_q$ are simple. For $\Ext^1$, we can use \cite[Proposition 5.1.2 and (5.3)]{VdB19} to reduce the computations to showing $\Ext^1_Y(\Oscr_p,\Oscr_q)=0=\Hom_Y(\Oscr_p,\Oscr_{\tau q})$. The latter follows $p$, $q$ and $\tau q$ are different points by assumption.

\end{proof}
We can now finish the main result of this section:
\begin{proof}[Proof of Theorem \ref{thm:cremona}]
Recall from \S \ref{sec:PresVdB} that $\gamma: A' \hookrightarrow A^{(2)}$ was constructed as follows:
\begin{eqnarray} \notag (A')_{m,n} = & \Hom_X(\Oscr_X(-2n),\Oscr_X(-2m) \otimes_{o_X} m_{\tau^{-m}d} \ldots m_{\tau^{-n+1}d}) & \\
\notag & \hookdownarrow & \\ \label{eq:constructioncremona} & A_{2m,2n} = \Hom_X(\Oscr_X(-2n),\Oscr_X(-2m)) & \end{eqnarray}
We construct $\gamma_2: A'' \hookrightarrow A$ with respect to the points $p,q$ as in Theorem \ref{thm:secondextension}. I.e.
\[ (A'')_{m,n} = \Hom_X(\Oscr_X(-n),\Oscr_X(-m) \otimes_{o_X} m_{d_m} \ldots m_{d_{n-1}}) \]
where 
\[ d_i = \begin{cases} 
\tau^{-j}p & \text{if $i=2j$} \\
\tau^{-j}q & \text{if $i=2j+1$}
\end{cases}
\]
By Lemma \ref{lem:msplits} for each $i$ we can write $m_{d_{2i}}m_{d_{2i+1}} = m_{\tau^{-i}(p+q)}$. In particular the inclusions $m_{\tau^{-i}(p+q+r)} \hookrightarrow m_{\tau^{-i}(p+q)}$ give rise to an inclusion $\gamma_1: A' \hookrightarrow A^{\prime \prime (2)}$ such that $\gamma = \gamma_2 \circ \gamma_1$. It hence remains to show that $\gamma_1$ is in fact an inclusion as in Theorem \ref{thm:firstextension}. For this we need to prove the existence of a point $p'$ such that
\begin{eqnarray} \label{eq:cremonafactor} \ \ \ \ \ \ \ \gamma_1((A')_{m,n}) = \Hom_{X'}(\Oscr_{X'}(-2n), \Oscr_{X'}(-2m) \otimes_{o_{X'}} m'_{\tau^{-m}p'} \ldots m'_{\tau^{-n+1}p'})\end{eqnarray}
where $m_{p'}' = \ker \left( o_{X'} \rightarrow o_{p'} \right)$. 
As $A'$ is generated in degree 1, it suffices to check \eqref{eq:cremonafactor} for $n=m+1$ in which case the left hand side of \eqref{eq:cremonafactor} equals \[ \Gamma(Y,\psi^{2m*}\Lscr \otimes \psi^{2m+1*}\Lscr \otimes \Oscr_Y(-\tau^{-m}p-\tau^{-m}q-\tau^{-m}r))\] and the right hand side equals 
\[ \Gamma(Y,\psi^{2m*}\Lscr \otimes \Oscr_Y(\tau^{-m}p) \otimes \psi^{2m+1*}\Lscr \otimes \Oscr_Y(-\tau^{-m}q) \otimes \Oscr_Y(\tau^{-m}p')) \]. Hence the theorem is proven by choosing $p'=r$.
\end{proof}
Inspired by Theorem \ref{thm:cremona} we make the following definition:
\begin{definition} \label{def:quadratictransform}
Let $A, A'$ be three dimensional Sklyanin algebras. An inclusion $A' \hookrightarrow A^{(v)}$ is called a quadratic transform if it can be written as a composition of inclusions as in Theorems \ref{thm:firstextension} and \ref{thm:secondextension}.
\end{definition}
\begin{remark}
By construction a quadratic transform always induces an isomorphism of function fields.
\end{remark}
It immediately follows from the definition that if $A' \hookrightarrow A^{(v)}$ is a quadratic transform, then $v=2^n$ for some nonnegative integer $n$. By construction our noncommutative versions of $\mathbb{P}^1 \times \mathbb{P}^1 \dashrightarrow \mathbb{P}^2$ and $\mathbb{P}^2 \dashrightarrow \mathbb{P}^1 \times \mathbb{P}^1$ as in Theorems \ref{thm:firstextension}, \ref{thm:secondextension} are quadratic transforms. Theorem \ref{thm:cremona} implies the noncommutative Cremona transform as in \S \ref{sec:PresVdB} is a quadratic transform as well. Similar to the construction in Figure \ref{fig:cremona} one can introduce a (commutative) cubic Cremona transform $\mathbb{P}^1 \times \mathbb{P}^1 \dashrightarrow \mathbb{P}^1 \times \mathbb{P}^1$ by blowing up two points in each copy of $\mathbb{P}^1 \times \mathbb{P}^1$, provided that these points do not lie on one ruling, see for example Figure \ref{fig:cremonacubic}. As this cubic Cremona factors through $\mathbb{P}^2$ using the classical birational transformation, there is an obvious noncommutative generalization $A' \hookrightarrow A^{(2)}$ where $A$ and $A'$ are cubic Sklyanin $\mathbb{Z}$-algebras. By definition this noncommutative cubic Cremona transformation is a quadratic transform as well.
\begin{figure}[ht]
\includegraphics[width=15cm]{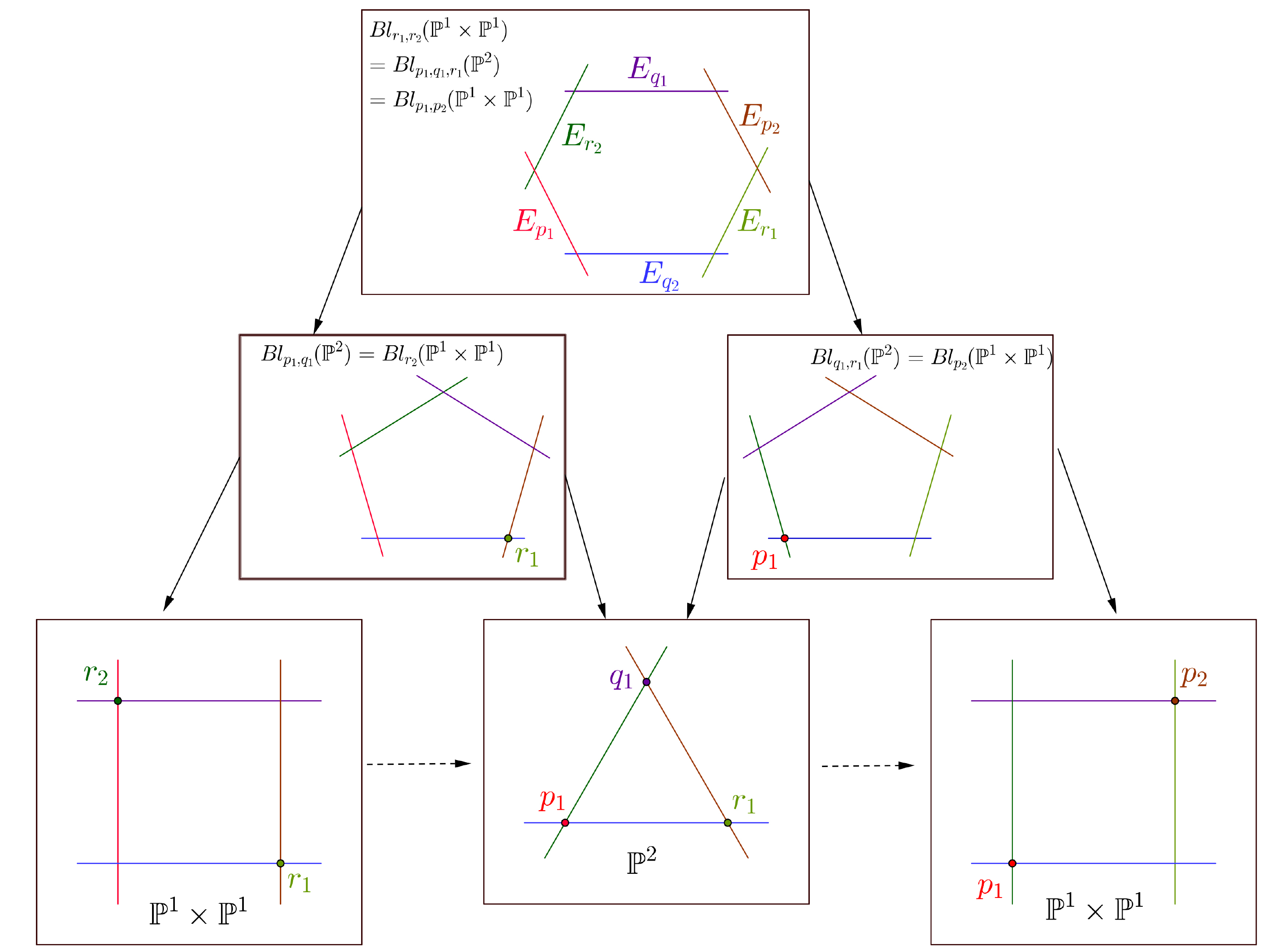}
\caption{The commutative cubic Cremona transform}
\label{fig:cremonacubic}
\end{figure}

\section{Inner morphisms}
\label{sec:innermorphisms}

The main goal of the following sections is to prove that in a suitable sense quadratic transforms are invertible:

\begin{definition} \label{def:inner}
Let A be a $\mathbb{Z}$-algebra such that $Q := \Frac(A)$ exists. 
We say that an injective morphism of $\mathbb{Z}$-algebras $\phi:A\hookrightarrow A^{(v)}$ is  \emph{inner} if there
exist $z_m\in Q_{vm,m}-\{0\}$ such that for $a\in A_{m,n}$ we have $\phi(a)=z_m a z_n^{-1}$. Moreover we require 
\begin{align}
\notag z_m&\in A_{vm,m} && \text{if $m<0$}\\
\label{eq:definner} z_0&=1\\
\notag z_m^{-1}&\in A_{m,vm} && \text{if $m>0$}
\end{align}
\end{definition}

The following is clear
\begin{proposition}
If $A\rightarrow A^{(v)}$ is inner then the induced map $Q_{0,0} \rightarrow (Q^{(v)})_{0,0}=Q_{0,0}$ is the identity.
\end{proposition}
\begin{definition} \label{def:invertible}
Let $A$ and $A'$ be $\mathbb{Z}$-algebras such that $\Frac(A)$ and $\Frac(A')$ exist. An inclusion $\gamma: A' \hookrightarrow A^{(w)}$ is said to be \emph{invertible} if there exists an inclusion $\delta: A \hookrightarrow A^{\prime (v)}$ such that $\delta \circ \gamma: A' \hookrightarrow A^{\prime (vw)}$ and $\gamma \circ \delta: A \hookrightarrow A^{(vw)}$ are inner.

\end{definition}
We can now state the main result of this paper:
\begin{theorem} \label{thm:inversemain} Assume that $\gamma: A'\hookrightarrow A^{(2^m)}$ is a quadratic transform between (three dimensional) Sklyanin $\mathbb{Z}$-algebras. Write $\gamma = \gamma_1 \circ \ldots \circ \gamma_t$ with the $\gamma_j$ as in Theorems \ref{thm:firstextension}, \ref{thm:secondextension}. Moreover we assume that for each factor $\gamma_i: A_i \hookrightarrow A_{i+1}^{(2)}$ (with $A_{i+1}$ necessarily cubic) the points $p_i, q_i$ used in the construction of $\gamma_i$ lie in different $\tau$-orbits. Then $\gamma$ is invertible and the ``inverse'' $\delta$ can be chosen as a quadratic transform $\delta: A\rightarrow A'^{(2^n)}$ with $|n-m| \leq 1$.
\end{theorem}

We will call $\delta$ as in the previous theorem an inverse quadratic transform to $\gamma$.

The following reduces the amount of work for proving Theorem \ref{thm:inversemain} dramatically.

\begin{lemma}
Let $\gamma_1: A' \hookrightarrow A^{(w_1)}$ and $\gamma_2: A'' \hookrightarrow A^{\prime (w_2)}$ be invertible as in definition \ref{def:invertible}. Then $\gamma_1 \circ \gamma_2$ is invertible as well.
\end{lemma}
\begin{proof}
By assumption there exist inclusion $\delta_1: A \hookrightarrow A^{\prime (v_1)}$ and $\delta_2: A' \hookrightarrow A^{\prime \prime (v_2)}$ such that $\delta_1 \circ \gamma_1$, $\delta_2 \circ \gamma_2$, $\gamma_1 \circ \delta_1$ and $\gamma_2 \circ \delta_2$ are inner. We now claim $(\delta_2 \circ \delta_1) \circ (\gamma_1 \circ \gamma_2)$ and $(\gamma_1 \circ \gamma_2) \circ (\delta_2 \circ \delta_1)$ are inner as well. As both proofs are analogous, we only prove the latter. Let $z_{1,m} \in \Frac{A}_{w_1v_1m,m}-\{0\}$ and $z_{2,m} \in \Frac{A'}_{w_2v_2m,m}-\{0\}$ be as in Definition \ref{def:inner}. Then for each $a \in A_{m,n}$:
\begin{eqnarray*}
((\gamma_1 \circ \gamma_2) \circ (\delta_2 \circ \delta_1))(a) & = & \gamma_1 \left( (\gamma_2 \circ \delta_2) (\delta_1(x)) \right) \\
& = & \gamma_1 \left( z_{2,v_1m} \delta_1(x) z_{2,v_1n}^{-1} \right) \\
& = & \gamma_1 (z_{2,v_1m}) (\gamma_1 \circ \delta_1)(x) \gamma_1(z_{2,v_1n})^{-1} \\
& = & \gamma_1 (z_{2,v_1m}) (z_{1,m} x z_{1,n}^{-1}) \gamma_1(z_{2,v_1n})^{-1} \\
& = & \left( \gamma_1 (z_{2,v_1m}) z_{1,m} \right) x \left( \gamma_1(z_{2,v_1n})z_{1,n}\right)^{-1}
\end{eqnarray*}
Moreover obviously
\begin{eqnarray*}
\gamma_1 (z_{2,v_2m}) z_{1,m} &\in & \Frac(A)_{w_1v_1w_2v_2m, m}\setminus \{0\} \\
\gamma_1 (z_{2,v_2m}) z_{1,m} &\in & A  \ \ \ \text{if $m<0$}\\
\gamma_1 (z_{2,0})) z_{1,0}&=& 1 \\
\left( \gamma_1 (z_{2,v_2m}) z_{1,m} \right)^{-1}&\in & A \ \ \ \text{if $m>0$}
\end{eqnarray*}
\end{proof}

The proof of Theorem \ref{thm:inversemain} then follows from the following 2 theorems which are proven in the next section:

\begin{theorem} \label{thm:inversemaincase1} Let $\gamma: A'\hookrightarrow A$ be a quadratic transform as in Theorem \ref{thm:secondextension}. Then $\gamma$ is invertible and the inverse can be chosen as $\delta: A \hookrightarrow A^{\prime(2)}$ as in Theorem \ref{thm:firstextension}.
\end{theorem}
\begin{theorem} \label{thm:inversemaincase2} Let $\gamma: A'\hookrightarrow A^{(2)}$ be a quadratic transform as in Theorem \ref{thm:firstextension} and assume that the points $p,q$ used in the construction of $\gamma$ lie in different $\tau$-orbits. Then $\gamma$ is invertible and the inverse can be chosen as $\delta: A \hookrightarrow A'$ as in Theorem \ref{thm:secondextension}.
\end{theorem}

%SO FAR THE FOLLOWING IS A CONJECTURE BUT WE MIGHT BE ABLE TO PROVE IT:
%A scalar automorphism $\psi:A\rightarrow A$ is an automorphims such that there exist $(\lambda_m)_m\in k^\ast$ such that for 
%$a\in A_{mn}$ we have $\psi(a)=\lambda_m  \lambda_n^{-1} a$.
%\begin{conjectures} $\delta$ is unique up to composition with a scalar automorphism of $A$.
%\end{conjectures}
%\begin{convention} \label{con:orbits}
%The construction of a quadratic transform depends on the choice of points on $Y$. In the following sections we will always assume these points lie in different $\tau$-orbits.
%\end{convention}
%TO DO: DEZE SECTIE ZAL HERSCHREVEN WORDEN. ZO ZIJN ER WEL QUADRATIC TRANSFORMS TUSSEN CUBISCHE SKLYANINS EN IS DE CREMONA GEFACTORISEERD DOOR EEN CUBISCHE SKLYANIN. WE ZOUDEN QUADRATIC TRANSFORM KUNNEN DEFINIEREN ALS EENDER WELKE SAMENSTELLING VAN KWADRATISCH-CUBISCH OF CUBISCH-KWADRATISCH?

%\medskip
\begin{remark}
\label{rem:approach}
Our approach to proving Theorems \ref{thm:inversemaincase1} and \ref{thm:inversemaincase2} is as follows: we first construct an inclusion of $\mathbb{Z}$-algebras $\delta: A \hookrightarrow A^{'(v)}$ such that the composition $\gamma \circ \delta: A \rightarrow A^{(wv)}$ is inner. Afterwards we will show that $\delta$ is in fact a quadratic transform and that $\delta \circ \gamma$ is inner as well.
\end{remark}

%Due to the technicalities in the construction we first consider the case where $A$ and $A'$ are quadratic in \S \ref{sec:invertibletransformquadratic}. Later we consider the case where $A$ is quadratic and $A'$ is a quadric or vice versa in \S \ref{sec:invertibletransformquadric}

\section{Inverting quadratic transforms between quadratic Sklyanin algebras and cubic Sklyanin $\mathbb{Z}$-algebras}
\label{sec:invertibletransformquadric}
 
In this section we finish the proof of Theorem \ref{thm:inversemain} by proving Theorem \ref{thm:inversemaincase1} and Theorem \ref{thm:inversemaincase2}. The proofs of these theorems are intertwined:% in case the quadratic transform $\gamma: A' \rightarrow A^{(w)}$ is a noncommutative version of $\mathbb{P}^1 \times \mathbb{P}^1 \dashrightarrow \mathbb{P}^2$ or $\mathbb{P}^2 \dashrightarrow \mathbb{P}^1 \times \mathbb{P}^1$ as in section \ref{sec:addendum}.

%For this we use an approach as in Remark \ref{rem:approach}: 
In \S \ref{sec:quadricinquadratic} we prove that if $\gamma: A' \hookrightarrow A$ is as in Theorem \ref{thm:secondextension}, then there is a quadratic transform $\delta: A \hookrightarrow A^{\prime (2)}$ such that $\gamma \circ \delta$ is inner. In \S \ref{sec:quadraticinquadric} we prove that if $\gamma': A' \hookrightarrow A^{(2)}$ is as in Theorem \ref{thm:firstextension}, then there is a quadratic transform $\delta': A \hookrightarrow A'$ such that $\gamma' \circ \delta'$ is inner. Finally in \S \ref{sec:invertabilitytransform} we prove that these constructions are each others inverses. I.e. if we were to construct $\delta$ out of $\gamma$ as in \S \ref{sec:quadricinquadratic}, set $\gamma' = \delta$ and compute $\delta'$ as in \S \ref{sec:quadraticinquadric}, then $\delta' = \gamma$. This allows us to conclude that not only $\gamma \circ \delta$, but also $\delta \circ \gamma$ is inner (as it is equal to $\gamma' \circ \delta'$). The analogous results are true if we were to start from $\gamma'$.

\subsection{The $\mathbb{Z}^2$-algebra associated to a noncommutative $\mathbb{P}^2 \dashrightarrow \mathbb{P}^1 \times \mathbb{P}^1$}
\label{sec:quadricinquadratic}

Throughout this subsection $\gamma: A' \hookrightarrow A$ will be a quadratic transform between a quadratic Sklyanin algebra $A = A(Y,\Lscr,\psi)$ and a cubic Sklyanin $\mathbb{Z}$-algebra $A'$. Recall from \S \ref{subsec:quadricinverse} that the construction of $\gamma$ is based on the choice of two points $p,q \in Y$. We will use the notation from this section, in particular $d_i$ is defined as in $(\ref{definitiondi})$. Moreover we define $\tau = \psi^3$, $\Lscr_i = \sigma^{*i}\Lscr$ and assume $p$ and $q$ lie different $\tau$-orbits.
%In this subsection we prove the existence of a quadratic transform $\delta: A \hookrightarrow A^{\prime (2)}$ such that $\gamma \circ \delta$ is inner. The proof that $\delta \circ \gamma$ is inner is postponed to subsection \ref{sec:invertabilitytransform} as this uses results from subsection \ref{sec:quadraticinquadric}.
%More concretely in subsection \ref{sec:quadraticinquadric} we construct a $\gamma'$ such that $\delta \circ \gamma'$ is inner and finally we prove $\gamma'=\gamma$.

We ``glue'' the algebras $A$ and $A'$ into a single $\mathbb{Z}^2$-algebra:

\[
\tilde{A}_{(i,j),(m,n)}:=
\begin{cases}
\Hom_X(\Oscr_X(-m-n),\Oscr_X(-i-j)m_{d_j}\ldots m_{d_{n-1}}&\text{if $n> j$}\\
\Hom_X(\Oscr_X(-m-n),\Oscr_X(-i-j))&\text{if $n \leq j$}
\end{cases}
\]

With $X$ and $\Oscr_X(i)$ as in Lemma \ref{lem:standardvanishing}.\\
Note that
\[
\tilde{A}_{(i,0),(m,0)}=A_{i,m}
\]
and
\[ \tilde{A}_{(0,j),(0,n)}  = \begin{cases} \Hom_X(\Oscr_X(-n),\Oscr_X(-j)m_{d_j}\cdots m_{d_{n-1}}) & \textrm{ if $n \geq j$} \\ 0 & \textrm{ if $n < j$} \end{cases} \]
such that
\[ \tilde{A}_{(0,j),(0,n)} = A'_{j,n} \]
In other words $\tilde{A}$ ``contains'' both $A$ and $A'$. 

\begin{remark}
By construction $\tilde{A}_{(i,j),(m,n)} \subset A_{i+j,m+n}$ and as such $\tilde{A}$ contains no nontrivial zero divisors.
\end{remark}

We now give some results on the dimensions of certain $\tilde{A}_{(i,j),(m,n)}$. For this let $h(n)$ be the Hilbert function of $A$ and $h'(n)$ be the Hilbert function of $A'$. I.e.
\begin{eqnarray} \label{eq:Hilbertdefinition}
\forall i \in \mathbb{Z}: \dim_k(A_{i,i+n})=h(n) \textrm{ and } \dim_k(A'_{i,i+n})=h'(n)
\end{eqnarray}

The following easy properties of $\tilde{A}$ are immediate from the definition:
\begin{proposition} Let $\tilde{A}$ be as above then
\begin{enumerate}
\item $\dim_k \tilde{A}_{(i,j),(i,j+b)}  = h'(b)$ holds for all $b,i,j \in \mathbb{Z}$ with $b \geq 0$ 
\item $\dim_k \tilde{A}_{(i,j),(i+a,j+b)}  = h(a+b)$ holds for all $a,b,i,j \in \mathbb{Z}$ with $a \geq 0, b \leq 0$
\end{enumerate}
\end{proposition}

More interesting is the following:

\begin{lemma} \label{lem:negativedim}
Let $a \leq 0$, then:
\begin{eqnarray}
\label{eq:negativedimension}
\dim_k \tilde{A}_{(i,j),(i+a,j+b)}=\dim_k \tilde{A}_{(i,j),(i,j+b+2\cdot a)}= h'(b+2 \cdot a)
\end{eqnarray}
\end{lemma}
\begin{proof}

The case $a=-1$ can be done analogously to \S \ref{subsec:quadricinverse} Step 5) and 6). For $a \leq -2$ we can no longer use Lemma \ref{cor:standardvanishing} and the proof is based on the existence of an ``$I$-basis'' (see \cite{TateVdB} for the definition and construction of an $I$-basis). As we will only use the case $a=-1$, we refer the interested reader to Appendix \ref{sec:Ibasis} for the details of the proof for $a \leq -2$.
\end{proof}

As a result of Lemma \ref{lem:negativedim} both $\tilde{A}_{(i,j),(i-1,j+2)}$ and $\tilde{A}_{(i,j),(i+1,j-1)}$  are one dimensional. Let $\delta_{i,j}$ and $\gamma_{i,j}$ be nonzero elements in these spaces. We can then visualise $\tilde{A}$ on a 2-dimensional square grid.

\begin{center}
\begin{tikzpicture}
\matrix(m)[matrix of math nodes,
row sep=4em, column sep=4em,
text height=1.5ex, text depth=0.25ex]
{(3,0) & (3,1) & (3,2) & (3,3) \\
(2,0) & (2,1) & (2,2) & (2,3) \\
(1,0) & (1,1) & (1,2) & (1,3) \\
(0,0) & (0,1) & (0,2) & (0,3) \\ };
\path[->,font=\scriptsize]
(m-1-1) edge  (m-1-2)
        edge[dotted] node[left]{$\delta_{3,0}$} (m-2-3)
(m-1-2) edge  (m-1-3)
        edge[dotted] node[left]{$\delta_{3,1}$}(m-2-4)
(m-1-3) edge  (m-1-4)
(m-2-1) edge  (m-2-2)
        edge[dotted] node[left]{$\delta_{2,0}$} (m-3-3)
        edge   (m-1-1)
(m-2-2) edge  (m-2-3)
        edge[dotted] node[left]{$\delta_{2,1}$} (m-3-4)
        edge[dotted] node[below]{$\gamma_{2,1}$} (m-1-1)
        edge   (m-1-2)
(m-2-3) edge  (m-2-4)
        edge   (m-1-3)
        edge[dotted] node[above]{$\gamma_{2,2}$} (m-1-2)
(m-2-4) edge   (m-1-4)
        edge[dotted] node[above]{$\gamma_{2,3}$} (m-1-3)
(m-3-1) edge  (m-3-2)
        edge[dotted] node[left]{$\delta_{1,0}$} (m-4-3)
        edge   (m-2-1)
(m-3-2) edge  (m-3-3)
        edge[dotted] node[left]{$\delta_{1,1}$} (m-4-4)
        edge   (m-2-2)
        edge[dotted] node[below]{$\gamma_{1,1}$} (m-2-1)
(m-3-3) edge  (m-3-4)
        edge   (m-2-3)
        edge[dotted] node[below]{$\gamma_{1,2}$}(m-2-2)    
(m-3-4) edge   (m-2-4)
        edge[dotted] node[above]{$\gamma_{1,3}$}(m-2-3)
(m-4-1) edge  (m-4-2)
        edge   (m-3-1)
(m-4-2) edge  (m-4-3)
        edge   (m-3-2)
        edge[dotted] node[below]{$\gamma_{0,1}$}(m-3-1)
(m-4-3) edge  (m-4-4)
        edge   (m-3-3)
        edge[dotted] node[below]{$\gamma_{0,2}$}(m-3-2)    
(m-4-4) edge   (m-3-4)
        edge[dotted] node[above]{$\gamma_{0,3}$}(m-3-3)   
;
\end{tikzpicture}
\end{center}

The vertical arrows represent three dimensional vector spaces, whereas the horizontal arrows represent two dimensional vector spaces and the dotted arrows represent one dimensional vector spaces (labeled by $\gamma_{i,j}$ and $\delta_{i,j}$).\\

Now consider the following diagram

\begin{center}
\begin{tikzpicture}
\matrix(m)[matrix of math nodes,
row sep=4em, column sep=4em,
text height=1.5ex, text depth=0.25ex]
{(i,j) &  & (i,j+b) & \\
& (i-1,j+2) &  & (i-1,j+b+2) \\ };
\path[->,font=\scriptsize]
(m-1-1) edge  (m-1-3)
        edge (m-2-4)
        edge[dotted] node[left]{$\delta_{i,j}$} (m-2-2)
(m-1-3) edge[dotted] node[right]{$\delta_{i,j+b}$}(m-2-4)
(m-2-2) edge  (m-2-4) 
;
\end{tikzpicture}
\end{center}

%\[
%\xymatrix@C+20mm{
%(a,b)\ar[r]^{\tilde{A}_{(a,b),(a,d)}}\ar@{.>}[d]_{kc_{a,b}}\ar[dr]|{\tilde{A}_{(a,b),(a-1,d+2)}}&(a,d)\ar@{.>}[d]^{kc_{a,d}}\\
%(a-1,b+2)\ar[r]_{\tilde{A}_{(a-1,b+2),(a-1,d+2)}}&(a-1,d+2)
%}
%\]
From Lemma \ref{lem:negativedim} we conclude that the vector spaces on the solid arrows all have the same dimension. 
Hence since $A$ is a domain we have an isomorphism of vector spaces:
\begin{equation}
\label{ref-9-11}
\delta_{i,j}^{-1}\cdot \delta_{i,j+b}: \tilde{A}_{(i,j),(i,j+b)}\rightarrow\tilde{A}_{(i-1,j+2),(i-1,j+b+2)}
\end{equation}
Whenever $2i+j = 2m + n$ and $i \geq m$ we write
%Put $p(a,b)=2a+b$, $q(a,b)=2b+a$. Assume $p(a,b)=p(c,d)$. If $d\ge b$ then we put 
\begin{eqnarray}
\label{eq:deltaijmn1}
\delta_{(i,j),(m,n)}=\delta_{i,j}\delta_{i-1,j+2}\ldots \delta_{m+1,n-2}\in \tilde{A}_{(i,j),(m,n)} 
\end{eqnarray}
when $i < m$ we define
%Let $\tilde{Q}$ be the quotient field of $\tilde{A}$. If $d\le b$ put
\begin{eqnarray}
\label{eq:deltaijmn2}
\delta_{(i,j),(m,n)}:=\delta^{-1}_{(m,n),(i,j)}\in \Frac(A)_{i+2j,m+2n}
\end{eqnarray}

In particular we always have
\[
\delta_{(i,j),(k,l)} \delta_{(k,l),(m,n)}=\delta_{(i,j),(m,n)}
\]
From \eqref{ref-9-11} we obtain an isomorphism
\[
\delta_{(0,2i+j),(i,j)}\cdot \delta_{(i,j+b),(0,2i+j+b)}:\tilde{A}_{(i,j),(i,j+b)}\rightarrow\tilde{A}_{(0,2i+j),(0,2i+j+b)}
\]
Now note that there is always an inclusion
\begin{equation}
\label{ref-10-12}
\cdot \delta_{(m,n),(i,n+2m-2i)}:\tilde{A}_{(i,j),(m,n)}\rightarrow\tilde{A}_{(i,j),(i,n+2m-2i)}
\end{equation}
If $m\geq i$ this follows from the fact that $A$ is a domain.
Lemma \ref{lem:negativedim} tells us the map is also well defined if $m<i$, in which case case \eqref{ref-10-12} even is an isomorphism.

Summarizing we obtain an inclusion
\[
\delta_{(0,2i+j),(i,j)}\cdot \delta_{(m,n),(0,2m+n)}:\tilde{A}_{(i,j),(m,n)}\rightarrow \tilde{A}_{(0,2i+j),(0,2m+n)}
\]
And hence an inclusion
\[
\delta_{(0,2i),(i,0)}\cdot \delta_{(m,0),(0,2m)}:A_{i,m}=\tilde{A}_{(i,0),(m,0)}\rightarrow \tilde{A}_{(0,2i),(0,2m)}=A'_{2i,2m}
\]
One easily checks that these inclusions are compatible with multiplication on $A$ and $A'$ such that we get an inclusion of algebras
\begin{eqnarray*}
\label{eq:theinverseinclusion}
\delta:A\hookrightarrow A^{\prime (2)}
\end{eqnarray*}

Our goal is to show that $\delta$ is in fact a quadratic transform as in Definition \ref{def:quadratictransform}. Moreover we want $\gamma \circ \delta$ to be inner (as stated in the introduction of this section, the proof to show that $\delta \circ \gamma$ is inner is postponed to \S \ref{sec:invertabilitytransform}). We first prove the latter:\\

Let $\gamma_{(i,j),(m,n)}$ be defined like $\delta_{(i,j),(m,n)}$ but using
$\gamma_{i,j}$ instead of $\delta_{i,j}$ (and $i+j=m+n$ in stead of $2i+j=2m+n$; recall $\gamma_{i,j}$ lies in the 1-dimensional space $\tilde{A}_{(i,j),(i+1,j-1)}$).\\  It is easy to see that the quadratic
transform $\gamma: A'\rightarrow A$ we started with is given by
\[
\gamma_{(j,0),(0,j)}\cdot \gamma_{(0,n),(n,0)}:A'_{j,n}=\tilde{A}_{(0,j),(0,n)}\rightarrow \tilde{A}_{(2j,0),(2n,0)}=A_{2j,2n}
\]
So the composition is given by
\begin{eqnarray}
\label{eq:innerquadratic1}
\gamma_{(2i,0),(0,2i)}\delta_{(0,2i),(i,0)}\cdot \delta_{(m,0),(0,2m)}\gamma_{(0,2m),(2m,0)}:A_{i,m}\rightarrow A_{2i,2m}
\end{eqnarray}
This composition is inner with
\[
z_i=\gamma_{(2i,0),(0,2i)}\delta_{(0,2i),(i,0)} \in \Frac(A)_{2i,i}
\]
One easily checks that the elements $z_i$ indeed satisfy the conditions in \eqref{eq:definner}.\\

We now prove that $\delta$ is a quadratic transform. For this we need to show the existence of a point $p' \in Y$ such that 

\begin{eqnarray}
\notag \delta(A_{i,i+1}) & = & \Gamma(Y, \Gscr_{2i} \Gscr_{2i+1} (-\tau^{-i} p') \\
\label{eq:existencep'} & = & \Gamma(Y,\Lscr_{2i}\Lscr_{2i+1}(-d_{2i}-d_{2i+1}-\tau^{-i}p'))
\end{eqnarray}

We first define $\tilde{B}$ like $\tilde{A}$ but starting from $B$ instead of from $A$. We find (using \eqref{eq:OX2} and \eqref{eq:identifyOY})
\[
\tilde{B}_{(i,j),(m,n)}:=
\begin{cases}
\Gamma(Y,\Lscr_{i+j} \Lscr_{i+j+1} \ldots \Lscr_{m+n}(-d_j - \ldots - d_{n-1} ))&\text{if $n> j$}\\
\Gamma(Y,\Lscr_{i+j} \Lscr_{i+j+1} \ldots \Lscr_{m+n})&\text{if $n\le j$}
\end{cases}
\]
Similar to \cite[Lemma 6.7]{PresVdB} one can show
\begin{lemma} The canonical map
\[
\tilde{A}_{(i,j),(m,n)}\rightarrow\tilde{B}_{(i,j),(m,n)}
\]
is an epimorphism in the first quadrant (i.e.\ $m\geq i$, $n \geq j$).
\end{lemma}

Recall that by the definition of $\delta$ we have for each $x\in A_{i,i+1}$
\begin{equation}
\label{ref-12-14-2}
\delta_{i,0}\delta_{i-1,2}\cdots \delta_{1,2i-2}\delta(x)=x \delta_{i+1,0} \delta_{i,2}\cdots \delta_{1,2i}
\end{equation}
when $i \geq 0$ and
\begin{equation}
\label{ref-12-14-2'}
\delta(x) \delta_{0,2i+2}\delta_{-1,2i+4}\ldots \delta_{i+2,-2}=\delta_{0,2i} \delta_{-1,2i+2} \ldots \delta_{i+1,-2} x 
\end{equation}
when $i < 0$. Hence in order to prove the existence of $p'$ in \eqref{eq:existencep'} we have to understand (the product of) the image(s) $\overline{\delta_{i,j}}$ of $\delta_{i,j}$ in 
\begin{equation}
\label{ref-13-15}
\tilde{B}_{(i,j),(i-1,j+2)}=\Gamma(Y,\Lscr_{i+j}(-d_j-d_{j+1}))
\end{equation}
As $\Lscr_{i+j}(-d_j-d_{j+1})$ has degree 1 on $Y$ we can choose a point $p'_{i,j}$ defined by
\begin{eqnarray} \label{eq:defpij} d_j + d_{j+1} + p'_{i,j} \sim [ \Lscr_{i+j} ] \end{eqnarray}
such that
\begin{equation} \label{eq:Bpij} \tilde{B}_{(i,j),(i-1,j+2)}=\Gamma(Y,\Lscr_{i+j}(-d_j-d_{j+1})) = \Gamma(Y,\Lscr_{i+j}(-d_j-d_{j+1}-p'_{i,j}))
\end{equation}

\begin{lemma} 
\label{lem:followingidentities}
Define $p'$ by the following identitiy
\begin{align*}
p+q+\tau p'&\sim [\Lscr_0]
\end{align*}
Then $p'_{i,j}=\tau^{-i+1}p'$.
\end{lemma}
\begin{proof}
As $\psi$ is a translation such that $\Lscr_{i+1} \cong \psi^* \Lscr_i$, there is an invertible sheaf $\Nscr$ of degree zero (see for example \cite[Theorem 4.2.3]{VdB38}) such that
\[ [\Lscr_n] = [\Lscr_0] + 3n [\Nscr] \]
and 
\[ \tau^{-i} p \sim p + 3i [\Nscr] \]
As $p'_{i,j}$ is uniquely defined by $(\ref{eq:defpij})$ this proves the lemma.
\end{proof}

In particular if $p'$ is as in the above lemma, then \eqref{eq:Bpij} gives rise to
\begin{equation} \label{eq:Bpij2} \tilde{B}_{(i,j),(i-1,j+2)}=\Gamma(Y,\Lscr_{i+j}(-d_j-d_{j+1})) = \Gamma(Y,\Lscr_{i+j}(-d_j-d_{j+1}-\tau^{-i+1}p'))
\end{equation}

%The following Lemma allows us to lift this result to $X$:

%\begin{lemma}
%\label{ref-4.4.2-16} Assume $e$ is an effective divisor of degree three on $Y$. If $e\sim [\Lscr_u]$ for some $u \in \mathbb{Z}$ then:
%\[ \dim_k \Hom_X(\Oscr_X(-u-1),\Oscr_X(-u)m_{e})=1 \]
%\end{lemma}
%\begin{proof}
%We have 
%\[ A_{u,u+1}=\Hom_X(\Oscr_X(-u-1),\Oscr_X(-u))\cong \Hom_Y(\Oscr_Y(-u-1),\Oscr_Y(-u))\cong B_{u,u+1}\]
%Moreover
%\[  \Hom_Y(\Oscr_Y(-u-1),\Oscr_Y(-u)(-e))=\Gamma(Y,\Lscr_u(-e)) \]
%So $\Hom_Y(\Oscr_Y(-u-1),\Oscr_Y(-u)(-e))$ is one-dimensional. Let $\partial$ be a non-zero section. Then it uniquely lifts to a non-zero morphism $\tilde{\partial} \in \Hom_X(\Oscr_X(-u-1),\Oscr_X(-u))$. We have to check that this lift actually lies in $\Hom_X(\Oscr_X(-u-1),\Oscr_X(-u)m_e)$. This is the case because 
%\[ \rho_Y \circ \partial = 0 \ \Rightarrow \rho_X \circ \tilde{\partial} = 0 \]
%where
%\[ \rho_Y: \Oscr_Y(-u) \rightarrow \Oscr_e(-u)  \]
%\[ \rho_X: \Oscr_X(-u) \rightarrow \Oscr_e(-u)  \]
%are the obvious morphisms
%which seems obvious.\dots
%\end{proof}

%Now recall that we assumed $p,q,r$ to lie in different $\tau$-orbits. This allows us to write
%\begin{eqnarray}
%\label{eq:splitideal}
%m_{Y,\tau^{-j}d}m_{Y,\tau^{-j-1}d}=m_{\tau^{-j}p+\tau^{-j}q} \cdot m_{\tau^{-j-1}p+\tau^{-j}r} \cdot m_{\tau^{-j-1}q+\tau^{-j-1}r}
%\end{eqnarray}
In particular  $\overline{\delta_{i,j}}$ is a non-zero section of $\Lscr_{i+j}(-d_j-d_{j+1}-\tau^{-i+1}p')$. As the latter has degree zero on $Y$ $\overline{\delta_{i,j}}$ is everywhere non-zero on $Y$.

In particular, going back to \eqref{ref-12-14-2} (and hence assuming $i \geq 0$, the case $i < 0$ being completely similar) we see that 
\[ \overline{\delta}_{i,0}\overline{\delta}_{i-1,2}\ldots \overline{\delta}_{1,2i-2} \]
is an everywhere non-zero section of
\[
\Lscr_{i} \ldots \Lscr_{2i-1}(-d_0-d_1-\tau^{-i+1}p'-d_2-d_3-\tau^{-i+2}p'-\ldots-d_{2i-2}-d_{2i-1}-p')
\]
Likewise 
\[
\overline{x} \ \overline{\delta}_{i+1,0}\overline{\delta}_{i,2}\ldots \overline{\delta}_{1,2i}
\]
is a section of
\[
\Lscr_{i} \ldots \Lscr_{2i+1}(-d_0-d_1-\tau^{-i}p'-d_2-d_3 -\tau^{-i+1}p'-\ldots
-d_{2i}-d_{2i+1} -p')
\]

so that $\overline{\delta(x)}$ is a section of $\Lscr_{2i}\Lscr_{2i+1}(-d_{2i}-d_{2i+1}-\tau^{-i}p')$. 
This is precisely what we had to show according to \eqref{eq:existencep'}.

\subsection{The $\mathbb{Z}^2$ algebra associated to a noncommutative  $\mathbb{P}^1 \times \mathbb{P}^1 \dashrightarrow \mathbb{P}^2$ }
\label{sec:quadraticinquadric}

Throughout this subsection $\gamma: A' \hookrightarrow A^{(2)}$ will be a quadratic transform between a cubic Sklyanin $\mathbb{Z}$-algebra $A = A(Y,(\Lscr_i)_{i \in \mathbb{Z}})$ and a quadratic Sklyanin algebra $A'$. Recall from \S \ref{subsec:quadric} that the construction of $\gamma$ is based on the choice of a points $p \in Y$. We will use the notation from this section, in particular $\tau = \alpha^2$. \\

We define the $\mathbb{Z}^2$-algebra $\tilde{A}$ as follows:

\[
\tilde{A}_{(i,j),(m,n)}:=
\begin{cases}
\Hom_X(\Oscr_X(-m-2n),\Oscr_X(-i-2j)m_{\tau^{-j}p}\ldots m_{\tau^{-n+1}p}&\text{if $n> j$}\\
\Hom_X(\Oscr_X(-m-2n),\Oscr_X(-i-2j))&\text{if $n \leq j$}
\end{cases}
\]

As in the previous section the following easy properties of $\tilde{A}$ are immediate from the definition.:
\begin{proposition} \label{prp:Atildequadric} Let $\tilde{A}$ be as above then
\begin{enumerate}
\item $\tilde{A}_{(i,0),(m,0)}=A_{i,m} $
\item $\tilde{A}_{(0,j),(0,n)} = A'_{j,n}$
\item $\tilde{A}$ contains no nontrivial zero divisors
\item $\dim_k \tilde{A}_{(i,j),(i,j+b)}  = h'(b)$ holds for all $b,i,j \in \mathbb{Z}$ with $b \geq 0$ and $h'$ the Hilbert series of $A'$
\item $\dim_k \tilde{A}_{(i,j),(i+a,j+b)}  = h(a+2b)$ holds for all $a,b,i,j \in \mathbb{Z}$ with $a \geq 0, b \leq 0$ and $h$ the Hilbert series of $A$
\end{enumerate}
\end{proposition}
Where for $(3)$ we used the fact that $A$ is a $\mathbb{Z}$-domain as in Theorem \ref{thm:quadricdomain}.

We also have the following partial analogue of Lemma \ref{lem:negativedim}:
\begin{lemma}
\label{lem:negativedimension2}
\begin{eqnarray*}
\dim_k \tilde{A}_{(i,j),(i-1,j+b)}=\dim_k \tilde{A}_{(i,j),(i,j+b-1)}= h'(b-1)
\end{eqnarray*}
\end{lemma}
\begin{proof}
The computation is completely similar to the $a=-1$ case of Lemma \ref{lem:negativedim} using Lemma \ref{cor:standardvanishing} and \cite[Lemma 6.5]{PresVdB}%as we can use \\
%$\Ext^1(\Oscr_X(-i),\Oscr_X(-i-1) \otimes \Dscr_{j,j+b})=0$ by Lemma \ref{cor:standardvanishing}.
\end{proof}
\begin{remark}
Although one cannot use an I-basis in the classical sense we expect $\dim_k \tilde{A}_{(i,j),(i+a,j+b)}= h'(b+2 \cdot a)$ to hold for all $a \leq 0$.
\end{remark}
As a corollary of Lemma \ref{lem:negativedimension2} and Proposition \ref{prp:Atildequadric}(3) we know $\tilde{A}_{(i,j),(i-1,j+2)}$ and $\tilde{A}_{(i,j),(i+1,j-1)}$  are one dimensional. Let $\delta_{i,j}$ and $\gamma_{i,j}$ be nonzero elements in these spaces. We can then visualise $\tilde{A}$ on a 2-dimensional square grid.

\begin{center}
\begin{tikzpicture}
\matrix(m)[matrix of math nodes,
row sep=4em, column sep=4em,
text height=1.5ex, text depth=0.25ex]
{(3,0) & (3,1) & (3,2) & (3,3) \\
(2,0) & (2,1) & (2,2) & (2,3) \\
(1,0) & (1,1) & (1,2) & (1,3) \\
(0,0) & (0,1) & (0,2) & (0,3) \\ };
\path[->,font=\scriptsize]
(m-1-1) edge  (m-1-2)
        edge[dotted] node[above]{$\delta_{3,0}$} (m-2-2)
(m-1-2) edge  (m-1-3)
        edge[dotted] node[above]{$\delta_{3,1}$}(m-2-3)
(m-1-3) edge  (m-1-4)
        edge[dotted] node[above]{$\delta_{3,2}$}(m-2-4)
(m-2-1) edge  (m-2-2)
        edge[dotted] node[left]{$\delta_{2,0}$} (m-3-2)
        edge   (m-1-1)
(m-2-2) edge  (m-2-3)
        edge[dotted] node[left]{$\delta_{2,1}$} (m-3-3)
        edge   (m-1-2)
(m-2-3) edge  (m-2-4)
        edge   (m-1-3)
        edge[dotted] node[left]{$\delta_{2,2}$} (m-3-4)
(m-2-4) edge   (m-1-4)
(m-3-1) edge  (m-3-2)
        edge[dotted] node[below]{$\delta_{1,0}$} (m-4-2)
        edge   (m-2-1)
(m-3-2) edge  (m-3-3)
        edge[dotted] node[below]{$\delta_{1,1}$} (m-4-3)
        edge   (m-2-2)
        edge[dotted] node[above]{$\gamma_{1,1}$} (m-1-1)
(m-3-3) edge  (m-3-4)
        edge   (m-2-3)
        edge[dotted] node[above]{$\gamma_{1,2}$}(m-1-2)  
        edge[dotted] node[below]{$\delta_{1,2}$} (m-4-4)  
(m-3-4) edge   (m-2-4)
        edge[dotted] node[above]{$\gamma_{1,3}$}(m-1-3)
(m-4-1) edge  (m-4-2)
        edge   (m-3-1)
(m-4-2) edge  (m-4-3)
        edge   (m-3-2)
        edge[dotted] node[below]{$\gamma_{0,1}$}(m-2-1)
(m-4-3) edge  (m-4-4)
        edge   (m-3-3)
        edge[dotted] node[below]{$\gamma_{0,2}$}(m-2-2)    
(m-4-4) edge   (m-3-4)
        edge[dotted] node[below]{$\gamma_{0,3}$}(m-2-3)   
;
\end{tikzpicture}
\end{center}
All horizontal arrows represent three dimensional  vector spaces whereas the vertical arrows represent two dimensional vector spaces and dotted arrows represent one dimensional vector spaces (labeled by $\gamma_{i,j}$ and $\delta_{i,j}$).

Completely identical to previous sections there is an inclusion
\[
\delta_{(0,i),(i,0)}\cdot \delta_{(m,0),(0,m)}:A_{i,m}=\tilde{A}_{(i,0),(m,0)}\rightarrow \tilde{A}_{(0,i),(0,m)}=A'_{i,m}
\]
(where the elements $\delta_{(i,j),(m,n)}$ are defined as in $(\ref{eq:deltaijmn1})$ and $(\ref{eq:deltaijmn2})$. The only thing which essentially changed is that $\delta_{(i,j),(m,n)}$ is now only defined when $i+j=m+n$ in stead of $i+2j=m+2n$.)\\
The induced inclusion of algebras
\begin{eqnarray*}
\label{eq:theinverseinclusion3}
\delta:A\hookrightarrow A'
\end{eqnarray*}
is such that the composition $\gamma \circ \delta$ is inner with
\[
z_i=\gamma_{(i,0),(0,2i)}\delta_{(0,i),(i,0)}
\]

Our next aim is to show that $\delta$ is a quadratic transform. For this we need to show the existence of two points $p', q' \in Y$ such that if we define $d'_i$ as
\begin{eqnarray} \label{definitiondi'}
d'_i = \begin{cases} 
\tau^{-l}p' & \text{if $i=2l$} \\
\tau^{-l}q' & \text{if $i=2l+1$} \\
\end{cases}
\end{eqnarray}
then we have
\begin{eqnarray}
\label{eq:existencep'q'} \delta(A_{i,i+1})  =  \Gamma(Y, \Gscr_{i} (-d'_i))  =  \Gamma(Y,\Lscr_{2i}\Lscr_{2i+1}(-\tau^{-i}p-d'_i))
\end{eqnarray}

We again start by defining a $\mathbb{Z}^2$-algebra $\tilde{B}$. This time it takes the following form:
\[
\tilde{B}_{(i,j),(m,n)}:=
\begin{cases}
\Gamma(Y,\Lscr_{i+2j} \Lscr_{i+2j+1} \ldots \Lscr_{m+2n}(-\tau^{-j}p - \ldots - \tau^{-n+1}p ))&\text{if $n> j$}\\
\Gamma(Y,\Lscr_{i+2j} \Lscr_{i+2j+1} \ldots \Lscr_{m+2n})&\text{if $n\leq j$}
\end{cases}
\]
Similar to \cite[Lemma 6.7]{PresVdB} one can show
\begin{lemma} The canonical map
\[
\tilde{A}_{(i,j),(m,n)}\rightarrow\tilde{B}_{(i,j),(m,n)}
\]
is an epimorphism in the first quadrant (i.e.\ $m\geq i$, $n \geq j$).
\end{lemma}

Recall that for each $x\in A_{i,i+1}$, $\delta(x)$ is related to $x$ and elements $\delta_{(i,j),(m,n)}$ via
\begin{equation}
\label{ref-12-14-3}
\delta_{i,0}\delta_{i-1,1}\cdots \delta_{1,i-1}\delta(x)=x \delta_{i+1,0} \delta_{i,1}\cdots \delta_{1,i}
\end{equation}
when $i \geq 0$ and
\begin{equation}
\label{ref-12-14-3'}
\delta(x) \delta_{0,i+1}\delta_{-1,i+2}\ldots \delta_{i+2,-1}=\delta_{0,i} \delta_{-1,i+1} \ldots \delta_{i+1,-1} x 
\end{equation}
when $i < 0$. Hence in order to understand $\delta(x)$ in $A'_{2i,2i+2} = B'_{2i,2i+2}$ we need to understand (the product of) the image(s) $\overline{\delta_{i,j}}$ of $\delta_{i,j}$ in $\tilde{B}$. First remark that if we choose $p', q' \in Y$ such that
\begin{eqnarray}
\notag p + \tau q' & \sim & [\Lscr_0] \\
\label{eq:defp'q'} p +  p' & \sim & [\Lscr_1]
\end{eqnarray}
then similar to Lemma \ref{lem:followingidentities} we then have for all $i,j$:
\[ \tau^{-j}p + d'_{i-1} \sim [\Lscr_{i+2j}] \]
with $d'_i$ as in \eqref{definitiondi'}.

%Lemma \ref{ref-4.4.2-16} then implies $\delta_{i,j}$ is a nonzero element in 
%\[ \Hom_X( \Oscr_X(-i-2j-1), \Oscr_X(-i-2j) m_{\tau^{-j}p+d'_{i-1}}) \]
%and 
$\overline{\delta_{i,j}}$ is a nonzero element of
\[ \Gamma( Y, \Lscr_{i+2j} (-\tau^{-j}p-d'_{i-1})) \]
As $\Lscr_{i+2j} (-\tau^{-j}p-d'_{i-1})$ has degree zero on $Y$, $\overline{\delta_{i,j}}$ is everywhere non-zero on $Y$.

In particular, going back to \eqref{ref-12-14-2} (and hence assuming $i \geq 0$, the case $i < 0$ being completely similar) we see that 
\[ \overline{\delta}_{i,0}\overline{\delta}_{i-1,1}\ldots \overline{\delta}_{1,i-1} \]
is an everywhere non-zero section of
\[
\Lscr_{i} \ldots \Lscr_{2i-1}(-p-d'_{i-1}-\tau^{-1}p-d'_{i-2}-\ldots-\tau^{-i+1}p-d'_0)
\]
Likewise 
\[
\overline{x} \ \overline{\delta}_{i+1,0}\overline{\delta}_{i,1}\ldots \overline{\delta}_{1,i}
\]
is a section of
\[
\Lscr_{i} \ldots \Lscr_{2i+1}(-p-d'_i-\tau^{-1}p -d'_{i-1}-\ldots
-\tau^{-i}p -d'_0)
\]

so that $\overline{\delta(x)}$ is a section of $\Lscr_{2i}\Lscr_{2i+1}(-\tau^{-i}p-d'_i)$. 
This is precisely what we had to show according to \eqref{eq:existencep'}.

\subsection{Invertability of the quadratic transforms}
\label{sec:invertabilitytransform}

We now show that if $\gamma$ and $\delta$ are as in \S \ref{sec:quadricinquadratic} or \S \ref{sec:quadraticinquadric}, then $\delta \circ \gamma$ is inner. This boils down to computations on the geometric data associated to $\gamma$ and $\delta$. First assume $A=A(Y,\Lscr,\psi)$ is quadratic and $\gamma: A' \hookrightarrow A$ is constructed with respect to $p,q \in Y$. Then according to Lemma \ref{lem:followingidentities} the quadratic transform $\delta: A \hookrightarrow A^{\prime (2)}$ is constructed with respect to a point $p' \in Y$ satisfying
\[ p + q+ \tau p' \sim [ \Lscr_0 ] \]
Using the techniques in \S \ref{sec:quadraticinquadric} we find a quadratic transform $\tilde{\gamma}: A' \rightarrow A$ such that $\tilde{gamma} \circ \delta$ is inner. By \eqref{eq:defp'q'} we know $\tilde{\gamma}$ is constructed with respect to points $p'', q'' \in Y$ satisfying

\begin{eqnarray*}
 p' + \tau q'' & \sim & [\Gscr_0] \\
 p' +  p'' & \sim & [\Gscr_1]
\end{eqnarray*}
 with $\Gscr_i$ as in \eqref{eq:Gscriquadratic}. Moreover \S \ref{sec:quadricinquadratic} constructs $A$ out of an elliptic helix $(\Lscr''_i)_{i \in \mathbb{Z}}$ given by
\[ \Lscr''_i = \Gscr_{2i} \otimes \Gscr_{2i+1} \otimes \Oscr_Y(-\tau^{-i} p) \]

We need to check that $p''=p, q''=q$ and $\Lscr''_i \cong \Lscr_i$. For this recall from \cite[Theorem 4.2.3]{VdB38} that there exists a linebundle $\Nscr$ of degree zero on $Y$ such that for each linebundle $\Mscr$ we have $[\psi^* \Mscr] = [\Mscr] + \deg(\Mscr) \cdot [\Nscr]$. Using this we find:
\begin{eqnarray*}
 p' + \tau q'' & \sim &  [ \Gscr_0] = [ \Lscr_0 \otimes \Oscr_Y(-p) ] \\
 & \Downarrow & \\
 p' + q'' -3[\Nscr] + p  & \sim & [\Lscr_0] \\
 & \Downarrow & \\
 p + q'' + \tau p' & \sim &  p + q + \tau p' \\
 & \Downarrow & \\
q'' & = & q
\end{eqnarray*}
Similarly $p''=p$. Next we show that the elliptic helix $(\Lscr_i'')_{i \in \mathbb{Z}}$ coincides with  $(\Lscr_i)_{i \in \mathbb{Z}}$:
\begin{eqnarray*}
[ \Lscr_i'' ] & = & [\Gscr_{2i}] + [\Gscr_{2i+1}] + [\Oscr_Y(-\tau^{-i}p')] \\
 & = & [\Lscr_{2i}] + [\Lscr_{2i+1}] + [\Oscr_Y(-d_{2i}-d_{2i+1})] + [\Oscr_Y(-\tau^{-i}p')] \\
 & = & [\Lscr_{2i}] + [\Lscr_{2i+1}] + [\Oscr_Y(-d_{i}-d_{i+1})] - 3i \cdot [\Nscr] + [\Oscr_Y(-\tau^{-i}p')] \\
 & = & \left( [\Lscr_{2i}] - 3i \cdot [\Nscr] \right) + \left([\Lscr_{2i+1}] + [\Oscr_Y(-d_{i}-d_{i+1} -\tau^{-i}p')] \right) \\ 
 & = & [\psi_*^i \Lscr_{2i}] = [\Lscr_i]
\end{eqnarray*}

Next we do similar computations in case $A = A(Y,(\Lscr_i)_{i \in \mathbb{Z}}$ is a quadric and $\gamma: A' \hookrightarrow A^{(2)}$ in constructed with respect to a point $p \in Y$ as in \S \ref{sec:quadricinquadratic}. Similar to the above it suffices to prove $p''=p$ and $\Lscr_i'' = \Lscr_i$ where
\begin{eqnarray*}
p'+q'+\tau p'' & \sim & [\Gscr_0] \\
\Lscr_i'' & = & \Gscr_i \otimes \Oscr_Y(-d_i')
\end{eqnarray*}
where $\Gscr_i$ is as in \eqref{eq:Gscriquadric} and $p',q',d_i'$ are as in \eqref{eq:defp'q'} and \eqref{definitiondi'}. First we prove $p''=p$, for this we take $\Nscr$ a degree zero linebundle on $Y$ such that $[\alpha^* \Mscr] = [\Mscr] + \deg(\Mscr) \cdot [\Nscr]$ with $\alpha$ as in Theorem \ref{thm:firstextension}

\begin{eqnarray*}
 p' + q' + \tau p'' & \sim &  [ \Gscr_0] = [ \Lscr_0 \otimes \Lscr_1 \otimes \Oscr_Y(-p) ] \\
 & \Downarrow & \\
 p' + q' + p'' - 2[\Nscr] + p  & \sim & [\Lscr_0] + [\Lscr_1] \\
 & \Downarrow & \\
 p' + \tau q' + p'' + p & \sim &  p + \tau q' + p + p' \\
 & \Downarrow & \\
p'' & = & p
\end{eqnarray*}
We now that the elliptic helix $(\Lscr_i'')_{i \in \mathbb{Z}}$ coincides with  $(\Lscr_i)_{i \in \mathbb{Z}}$:
\begin{eqnarray*}
[ \Lscr_i'' ] & = & [\Gscr_{i}] + [\Oscr_Y(-d_i)] \\
 & = & [\Lscr_{2i}] + [\Lscr_{2i+1}] + [\Oscr_Y(-d_i)] + [\Oscr_Y(-\tau^{-i}p')] \\
 & = & [\Lscr_{i}] + [\Lscr_{3i+1}] + [\Oscr_Y(-d_i)] + [\Oscr_Y(-\tau^{-i}p')] \\
 & = & [\Lscr_{i}] 
\end{eqnarray*}
Finishing the proof of Theorem \ref{thm:inversemain}.

\newpage
\appendix
\addtocontents{toc}{\protect\setcounter{tocdepth}{1}}

\section{Quadrics admit $\mathbb{Z}$-fields of fractions}
\label{sec:appendix}
%We know generalize the result in \cite{ATV2} that regular algebras of dimension 3 are domains to the case of quadrics. From this we can easily proof that quadrics admit $\mathbb{Z}$-fields of fractions. % As a quadric is a $\mathbb{Z}$-algebra by definition we should first introduce a suitable $\mathbb{Z}$-version of a domain.
In this appendix we prove the following
%The goal of this paper is to prove the following
\begin{theorem*}[Theorem \ref{thm:quadricdomain} and \ref{thm:quadricsfunctionfield}]
Let $A$ be a quadric, then $A$ is a $\mathbb{Z}$-domain and $A$ admits a $\mathbb{Z}$-field of fractions.
\end{theorem*}

The proof of this theorem is based on several preliminary results. 

\begin{notation}
Throughout this appendix $A$ will always be a quadric. Moreover for any $A$-module $M$ we let $\pd(M)$ and $\gkdim(M)$ denote the projective and Gelfand-Kirillov dimension respectively.
\end{notation}

\subsection{Preliminary results}
\subsubsection{Some lemmas}

\begin{lemma} \label{lem:pdimpliesgkdim}
Let $M$ be a finitely generated left- or right-$A$-module and assume $\pd(M) \leq 1$, then $\gkdim(M) \geq 2$.
\end{lemma}
\begin{proof}
Upon replacing the projective modules $A(-i)$ by $e_iA$ or $Ae_i$ one can copy the proof of \cite[Proposition 2.41]{ATV2}
\end{proof}

\begin{lemma} \label{lem:trivialsocle}
Let $M$ be a finitely generated right-$A$-module and let $S_i$ denote the simple module $e_iA / e_iA_{>i}$, then
\[ pd(M) \leq 2 \ \Rightarrow \ \forall i \in \mathbb{Z}: \Hom_A(S_i,M) = 0 \]
\end{lemma}
\begin{proof}
Upon replacing the projective modules $A(-i)$ by $e_iA$ one can copy the proof of \cite[Proposition 2.46 (i)]{ATV2}
\end{proof}

\begin{lemma} \label{lem:d-length}
Let $i \in \mathbb{Z}$ be any integer and $M$ be some graded submodule of $e_iA$ then $\gkdim(M) = 3 \ \Leftrightarrow \ \gkdim(e_iA/M) < 3$.
\end{lemma}
\begin{proof}
$\gkdim(e_iA/M) < 3 \Rightarrow gkdim(M) = 3$ is trivial. Let us prove the other direction. \\
Assume by way of contradiction that $\gkdim(M)=\gkdim(e_iA/M)=3$. As both $M$ and $e_iA/M$ are nonzero we have $e(M)>0$ and $e(e_iA/M)>0$. However as they have equal $\gkdim$, we have $e(e_iA) = e(M) + e(e_iA/M)$. A direct computation shows that $e(e_jA)=\frac{1}{2}$ holds for all $j$. Similar to the proof of \cite[Proposition 2.21 (iii)]{ATV2} we then know that $e(M)$ and $e(e_iA/M)$ must be a nonnegative multiple of $\frac{1}{2}$. Contradiction!
\end{proof}

We now introduce a homogeneous ideal $N$ of $A$ in a similar fashion as was done in \cite{ATV2}:

\begin{enumerate}
\item $e_iA$ is a Noetherian object in $\Gr(A)$. In particular any ascending chain of submodules of $e_iA$ must stabilize. %By Zorn's Lemma this
This allows us to set $N_i$ to be the largest submodule of $e_iA$ of $\gkdim \leq 2$. 
\item Define $N$ as $\displaystyle \bigoplus_{i \in \mathbb{Z}} N_i$. Then $N$ is a homogeneous two-sided ideal of $A$. To see why $N$ also has the structure of a left ideal, note that if $a \in A_{i,j}$ then $aN_j$ is a submodule of $e_i A$ of $\gkdim \leq 2$. This implies that $aN_j \subset N_i$ for otherwise $aN_j + N_i$ would be a strictly larger submodule than $N_i$ but it still has $\gkdim \leq 2$.
\end{enumerate}

\begin{remark} \label{rem:N2periodic}
Recall that $A$, being a quadric, is 2-periodic \cite[Proposition 5.6.1]{VdB38}. I.e. there is an isomorphism $A \cong A(2)$. This isomorphism induces an isomorphism $N \cong N(2)$.\\
To see this, fix any $i \in \mathbb{Z}$ and let $f_i: e_iA \rightarrow e_{i+2}A(2)$ be the induced isomorphism. Then $f_i(e_iN)$ has $\gkdim = 2$ in particular, being an $A(2)$-submodule of $e_{i+2}A(2)$ we must have $f_i(e_iN) \subset e_{i+2}N(2)$. By considering $f_i^{-1}$ we see that this must in fact be an equality. 
\end{remark}

\begin{lemma} \label{lem:ANdomain}
Let $N$ be as above, then $\overline{A} := A/N$ is a $\mathbb{Z}$-domain.
\end{lemma}
\begin{proof}
Let $b \in A_{i,j} \setminus N_{i,j}$, we then need to show that the induced morphism
\[ e_j \overline{A} \xrightarrow{ \overline{b} \cdot } e_i \overline{A} \]
is injective. For this consider the commutative diagram

\begin{center}
\begin{tikzpicture}
\matrix(m)[matrix of math nodes,
row sep=3em, column sep=3em,
text height=1.5ex, text depth=0.25ex]
{0 & \ker(b \cdot) & e_jA & bA & 0 \\
 0 & \ker(\overline{b} \cdot) & e_j \overline{A} & \overline{b} \overline{A} & 0 \\};
\path[->,font=\scriptsize]
(m-1-1) edge (m-1-2)
(m-1-2) edge (m-1-3)
(m-1-3) edge (m-1-4)
(m-1-4) edge (m-1-5)
(m-2-1) edge (m-2-2)
(m-2-2) edge (m-2-3)
(m-2-3) edge (m-2-4)
(m-2-4) edge (m-2-5);
\path[->>,font=\scriptsize]
(m-1-2) edge (m-2-2)
(m-1-3) edge (m-2-3)
(m-1-4) edge (m-2-4)
;
\end{tikzpicture}
\end{center}
Now suppose by way of contradiction that $\ker(\overline{b} \cdot) \neq 0$. By construction $e_j\overline{A} = e_jA/N_j$ does not contain submodules of $\gkdim \leq 2$, hence $\gkdim \left( \ker(\overline{b} \cdot) \right) = 3$. This implies that $\gkdim \left( \ker(b \cdot) \right) = 3$ as well. By Lemma \ref{lem:d-length} we must have $\gkdim \left( bA \right) < 3$, hence also $\gkdim \left( \overline{b}\overline{A} \right) < 3$. As $\overline{b}\overline{A} \subset e_i \overline{A}$ and $e_i \overline{A}$ does not contain submodules of $\gkdim \leq 2$ we must have $\overline{b} \overline{A} =0$, contradicting the fact that $\overline{b} \neq 0$.
\end{proof}

\subsubsection{Dual modules}
Next we introduce the notion of dualization of (right-)$A$-modules. Throughout this section we will use $A^{op}$ to denote the opposite algebra of $A$. $A^{op}$ is a $\mathbb{Z}$-algebra by setting $(A^{op})_{i,j} = (A_{-j,-i})^{op}$. With this $\mathbb{Z}$-algebra structure graded right-$A^{op}$-modules can be identified with graded left-$A$-modules; for example $e_i A^{op}$ naturally corresponds to $Ae_{-i}$. It hence makes sense to let $\Gr(A^{op})$ denote the categories of graded left $A$-modules.\\
Let $M$ be a graded right-$A$-module, then $\displaystyle \bigoplus_{i \in \mathbb{Z}} \Hom_{\Gr(A)}(M,e_i A)$ naturally has the structure of a graded left-$A$-module via:

\[ A_{ij} \otimes \Hom_{\Gr(A)}(M,e_j A) \rightarrow  \Hom_{\Gr(A)}(M,e_i A): x \otimes f \mapsto x \cdot f \]
where
\[ (x \cdot f)(m) := x \cdot f(m) \]
We denote this graded left-$A$-module by $M^*$ or $\Hom_{\Gr(A)}(M,A)$ and it is called the dual of $M$. One easily checks that this induces a left-exact functor $\Hom_{\Gr(A)}(-,A) = (-)^*: \Gr(A) \rightarrow \Gr(A^{op})$.\\
Note that as $\Hom_{\Gr(A)}(e_iA,e_j A) \cong A_{ji}$ we naturally have $\Hom_{\Gr(A)}(e_iA,A) = (e_iA)^* = Ae_i$. This allows us to define the right derived functors %of $\Hom_{\Gr(A)}(-,A)$ as follows: We first introduce a derived functor 
$\RHom_{\Gr(A)}(-,A): D^b_f(\Gr(A)) \rightarrow D^b_f(\Gr(A^{op}))$. If $C^\bullet$ is some object in $D^b_f(\Gr(A))$ which is represented by a bounded exact complex of finitely generated projectives, say
\[  0 \rightarrow \bigoplus_{i \in \mathbb{Z}} e_iA^{\oplus l_{i,n}} \xrightarrow{\cdot M_n} \bigoplus_{i \in \mathbb{Z}} e_iA^{\oplus l_{i,n-1}} \rightarrow \ldots \rightarrow \bigoplus_{i \in \mathbb{Z}} e_iA^{\oplus l_{i,m}}  \rightarrow 0 \]
where $M_n$ is some %(possibly infinite) 
matrix whose entries are homogeneous elements in $A$, then $\RHom_{\Gr(A)}(C^\bullet,A)$ is represented by the complex
\[ 0 \leftarrow \bigoplus_{i \in \mathbb{Z}} Ae_i^{\oplus l_{i,n}} \xleftarrow{ M_n \cdot } \bigoplus_{i \in \mathbb{Z}} Ae_i^{\oplus l_{i,n-1}} \leftarrow \ldots \leftarrow \bigoplus_{i \in \mathbb{Z}} Ae_i^{\oplus l_{i,m}}  \leftarrow 0 \]
(where each term in position $j$ in the original complex gives rise to a term in position $-j$ in the new complex)
Similar to the graded case we use the shorthand notation $(C^\bullet)^D := \RHom_{\Gr(A)}(C^\bullet,A)$
If $M$ is some graded right-$A$-module, then we denote $\Ext^i_{\Gr(A)}(M,A) := R^i \Hom_{\Gr(A)}(M,A) = h^i \left( M^D \right)$.

\begin{remark}
%One easily checks that $\RHom_{\Gr(A)}(-,A)$ is well defined. Moreover if
If we introduce $\RHom_{\Gr(A^{op})}(-,A): D^b_f(\Gr(A^{op})) \rightarrow D^b_f(\Gr(A))$ in an analogous way, then $((-)^D)^D \cong Id$ holds, giving rise to a biduality spectral sequence as in the graded case.
\end{remark}

For a bounded complex $C^\bullet$ of (finitely generated, graded right-)$A$-modules (or $A^{op}$-modules) we define the Hilbert series of $C^\bullet$ as
\[ h_{C^\bullet}(t) = \sum_{i \in \mathbb{Z}} h_i(C^\bullet) t^i \textrm{ with } h_i(C^\bullet) = \sum_{j \in \mathbb{Z}} (-1)^j \dim_k \left( (C^j)_i \right) \]
and we denote $e(C^\bullet)$ to be the leading coefficient of the series expension of $h_{C^\bullet}(t)$ in terms of $(1-t)^{-1}$ and $\gkdim(C^\bullet)$ as the highest power of $(1-t)^{-1}$ in this expansion, i.e. the order of pole of $h_{C^\bullet}(t)$

We then have the following:

\begin{lemma} \label{lem: Hilbertdualcomplex1}
Let $C^\bullet \in D^b_f(\Gr(A))$, then we have the following equality of rational functions:
\[ h_{(C^\bullet)^D}(t) = - t^4 \cdot h_{C^\bullet}(t^{-1}) \]
\end{lemma}
\begin{proof}
By linearity of the definition of $h_{C^\bullet}$, it suffices to prove the equality in case $C^\bullet$ is given by some projective $e_iA$ concentrated in position $j$. In this case $(C^\bullet)^D$ is given by $Ae_i$ (hence $e_{-i}A^{op}$) concentrated in position $(-j)$ such that
\begin{eqnarray*} h_{(C^\bullet)^D}(t) = (-1)^{-j} \cdot \frac{t^{-i}}{(1-t^2)(1-t)^2} & = & (-1)^j \cdot \frac{(t^{-1})^i}{-t^{4}(1-(t^{-1})^2)(1-t^{-1})^2}\\
& = &  -t^{-4} \cdot \frac{(t^{-1})^i}{(1-(t^{-1})^2)(1-t^{-1})^2} \\ &  = & -t^{-4} \cdot h_{C^\bullet}(t^{-1})
\end{eqnarray*}

\end{proof}

\begin{corollary} \label{lem:Hilbertdualcomplex}
Let $C^\bullet$ be a bounded complex of (finitely generated) right $A$-modules. Let $m = \gkdim(C^\bullet)$, then
\begin{enumerate}
\item $\gkdim((C^\bullet)^D) = m$
\item $e((C^\bullet)^D) = (-1)^{m+1} e(C^{\bullet})$
\end{enumerate}
\end{corollary}
\begin{proof}
Suppose 
\[ h_{C^\bullet}(t) = \frac{\sum_{i=0}^\infty \alpha_i (1-t)^i}{(1-t)^m} \]
with $\alpha_0 = e(C^{\bullet}) \neq 0$. Then we need to show that 
\[ h_{(C^\bullet)^D}(t) = \frac{\sum_{i=0}^\infty \widetilde{\alpha_i} (1-t)^i}{(1-t)^m} \]
with $\widetilde{\alpha_0} = (-1)^{m+1} \alpha_0$.\\ \\

First note that for each $n \in \mathbb{Z}$ we can write:
\[ t^n = 1 + \sum_{j =1}^\infty \beta_{n,j} (1-t)^j \]
with $\beta_{n,j} \in k^*$. Then by Lemma \ref{lem: Hilbertdualcomplex1} we have

\begin{eqnarray*}
 h_{(C^\bullet)^D}(t) & = & \frac{\sum_{i=0}^\infty \alpha_i (-t^{-4})(1-t^{-1})^i}{(1-t^{-1})^m} \\
& = & \frac{\sum_{i=0}^\infty - \alpha_i (-1)^i t^{-i-4} (1-t)^i}{(-1)^m t^{-m} (1-t)^m} \\
& = & \frac{\sum_{i=0}^\infty \alpha_i (-1)^{m+i+1} t^{m-i-4} (1-t)^i}{(1-t)^m} \\
& = & \frac{\sum_{i=0}^\infty \widetilde{\alpha_i} (1-t)^i}{(1-t)^m} \\
\end{eqnarray*}
where 
\[ \widetilde{\alpha_i} = (-1)^{m+i+1} \left( \alpha_i + \sum_{j=1}^i (-1)^j \beta_{m-i+j-4,j} \alpha_{i-j} \right) \]
\end{proof}

\begin{lemma} \label{lem:torsionext3}
Let $M$ be a finitely generated right-$A$-module, then $\Ext_{\Gr(A)}^3(M,A)$ is a finite dimensional $k$-vectorspace.
\end{lemma}
\begin{proof}
This is an immediate generalization of \cite[Proposition 2.46(ii)]{ATV2}.
\end{proof}

We now prove some more results on the homogeneous ideal $N$ as above:

\begin{lemma} \label{lem:leftgkdim}
For each $i \in \mathbb{Z}$ we have $\gkdim(Ne_i) \leq 2$.
\end{lemma}
\begin{proof}
By construction we know that for each $i$ we have $\gkdim(e_iN) \leq 2$. In particular there is for each $i$ a degree 2 polynomial $P_i$ such that $\dim_k(N_{i,i+l}) \leq P_i(l)$ holds for all $l$ sufficiently large. Now fix some $i \in \mathbb{Z}$, we must show that there is a degree 2 polynomial $Q$ such that $\dim_k(N_{i-l,i}) \leq Q(l)$. For this recall that the 2-periodicity of $A$ descends to $N$ (see Remark \ref{rem:N2periodic}). In particular we have
\[ \dim_k(N_{i-l,i}) = \begin{cases} \dim_k(N_{i,i+l}) & \textrm{if } l \textrm{ is even} \\ \dim_k(N_{i+1,i+l+1}) & \textrm{if } l \textrm{ is odd} \end{cases} \]
Without loss of generality we can now assume that $P_i(l) \geq P_{i+1}(l)$ holds for all $l$ sufficiently large. We can finish the proof by setting $Q=P_i$.

\end{proof}

\begin{lemma} \label{lem:annihilatorgkdim}
Fix some $i \in \mathbb{Z}$ and let $I \subset Ae_i$ be the left-annihilator of $e_i N$, then $\gkdim(I) =3$.
\end{lemma}
\begin{proof}

As $e_i N$ is a submodule of the noetherian right $A$-module $e_i A$, it is finitely generated. I.e. there are elements $x_{i,i} \in A_{i,i}, x_{i,i+1} \in A_{i,i+1}, \ldots , x_{i,n} \in A_{i,n}$ such that
\[ e_i N = \sum_{j=i}^n x_{i,j} A = \sum_{j=i}^n x_{i,j} e_j A \]
Let $I_j$ be the left annihilator of $x_{i,j}$, i.e.
\[ I_j = \{ a \in Ae_i \mid ax_{i,j} = 0 \]
Then there is an exact sequence of left $A$-modules
\[ 0 \rightarrow I_j \rightarrow Ae_i \xrightarrow{ \cdot x_{i,j}} Ae_j \]
such that $Ae_i / I_j \cong A x_{i,j}$. Moreover $I = \bigcap_j I_j$ such that
\[ \gkdim(A/I) \leq max_j \gkdim(A x_{i,j}) \leq max_j \gkdim(A e_j) = 2 \]
The result now follows from Lemma \ref{lem:d-length}.
\end{proof}

\begin{lemma} \label{lem:syzygy}
For each $i \in \mathbb{Z}$ we have:
\begin{enumerate}[(i)]
\item $e_i N$ is a second syzygy
\item $pd(e_iN) \leq 1$
\end{enumerate}
\end{lemma}
\begin{proof}
$(ii)$ obviously follows from $(i)$, so we only need to prove $e_i N$ is a second syzygy. By Lemmas \ref{lem:annihilatorgkdim} and \ref{lem:leftgkdim} we know there exists an element $b \in A_{ji}$ such that $bN=0$ while $b \not \in Ne_i$. Hence $e_iN \subset \ker(b \cdot)$ while $e_i \overline{A} \xrightarrow{ b \cdot} e_j \overline{A}$ is injective by Lemma \ref{lem:ANdomain}. This implies that we have a left exact sequence
\[ 0 \rightarrow e_iN \rightarrow e_iA \xrightarrow{b \cdot} e_jA \]
finishing the proof.
\end{proof}
\subsection{Proof of Theorem \ref{thm:quadricdomain}}

Let $N$ be as above. By Lemma \ref{lem:ANdomain} it suffices to prove that $N=0$. Suppose by way of contradiction that this is not the case. Without loss of generality we can assume $e_0 N \neq 0$. Then by Lemma \ref{lem:syzygy} $pd(e_0N) \leq 1$ which by Lemma \ref{lem:pdimpliesgkdim} implies $\gkdim(e_0N) \geq 2$. As by construction $\gkdim(e_0N) \leq 2$, we have $\gkdim(e_0N) = 2$. Let $(e_0N)^D = \RHom(e_0N, A)$ denote the dual complex as above, by the projective dimension of $e_0N$, this complex only has homology at position 0 and 1. By Lemma \ref{lem:syzygy} $e_0N$ is a second syzygy and hence we have $h^1((e_0N)^D) = \Ext^1(e_0N,A) \cong \Ext^3(M,A)$ for some module $M$.\\
Lemma \ref{lem:torsionext3} then implies that  $h^1((e_0N)^D)$ is finite dimensional. In particular the Gelfand-Kirillov dimension and multiplicity of $(e_0N)^D$ are solely determined by $(e_0N)^*$. Corollary \ref{lem:Hilbertdualcomplex} gives $e((e_0N)^*) = e((e_0N)^D) = -e (e_0N)$. A contradiction!

\subsection{Proof of Theorem \ref{thm:quadricsfunctionfield}}
By Theorem \ref{thm:quadricdomain} and Proposition \ref{prp:channyman} it suffices to prove that all $e_iA$ are uniform modules. For this fix any $i$ and nonzero $M \subset e_iA$. Then $\gkdim(M)=3$. To see this let $x$ be any nonzero element in $M_j \subset A_{i,j}$, then by Theorem \ref{thm:quadricdomain} we have
\[ 3 = \gkdim(e_jA) = \gkdim(x e_j A) \leq \gkdim(M) \leq \gkdim(e_iA) = 3 \]
Now let $N$ be any other nonzero submodule of $e_iA$. Then obviously $\gkdim(N)=3$ as well. Suppose by way of contradiction that $M \cap N = 0$, then the following composition is a monomorphism:
\[ N \hookrightarrow e_iA \rightarrow e_iA/M \]
such that $\gkdim(e_iA/M)=3$. This gives a contradiction with Lemma \ref{lem:d-length}. Hence for any nonzero $M,N \subset e_A$ we must have $M \cap N \neq 0$, s that $e_iA$ is a uniform module.

\section{$I$-bases for quadratic Sklyanin algebras and Lemma \ref{lem:negativedim}}
\label{sec:Ibasis}
Throughout this section we assume $A = A(Y,\Lscr,\psi)$ is a quadratic Sklyanin algebra with Hilbert series $h$. $p$ and $q$ are points lying in different $\tau$-orbits with $\tau = \psi^3$. Our goal is to prove that for $a \leq -2:$
\begin{eqnarray} \notag \dim_k \left( \Hom(\Oscr_X(-i-a-j-b), \Oscr_X(-i-j) \otimes m_{d_j} \ldots m_{d_{j+b-1}} ) \right) = \\
\label{eq:h2a+b} h'(2a+b) = \begin{cases} (n+1)^2 & \textrm{ if } 2a+b = 2n \geq 0 \\ (n+1)(n+2) & \textrm{ if } 2a+b = 2n+1 > 0 \\ 0 & \textrm{ if } 2a+b < 0 \end{cases} \end{eqnarray}
where $d_i$ is as in \eqref{definitiondi}. Using $\left( \Oscr_X(n) \otimes m_p \right)(m) = \left(\Oscr_X\otimes m_{\psi^{-n} p}\right)(n+m)$ (see for example \cite[\S 6]{PresVdB}) and replacing $p$ and $q$ by $\psi^x p$, $\psi^y q$ for the appropriate values of $x$ and $y$ this is equivalent to proving
\begin{eqnarray} \label{eq:Ibasisproof} \dim_k \left( \Hom(\Oscr_X, \left( \Oscr_X \otimes m_{d_0} \ldots m_{d_{b-1}} \right)(a+b) ) \right) = h'(2a+b) \end{eqnarray}
We will prove this using $I$-bases.
\subsection{$I$-bases}
In this subsection we recall the definition and construction of an $I$-basis for a quadratic Sklyanin algebra. For a more thorough introduction to $I$-bases we refer the reader to \cite{TateVdB}.

\begin{definition}
Let $A = A(Y,\Lscr,\psi)$ be a quadratic Sklyanin algebra and let $G$ denote the monoid of monomials in $x,y,z$. Let $G_n$ denote the subset of all degree $n$ monomials. An $I$-basis for $A$ is then given by a map $v: G \rightarrow A$ satisfying the following properties:
\begin{enumerate}[i)]
\item $v(G_n)$ is a $k$-basis for $A_n$
\item for any $g \in G$ there are elements $x_g, y_g, z_g \in A_1$ such that 
\[ v(gx)=v(g)x_g , v(gy)=v(g)y_g \textrm{ and } v(gz)=v(g)z_g \]
\end{enumerate}
\end{definition}
\begin{remark} \label{rem:scalarmultiple}
Note that $v(x)=x_1, v(y)=y_1, v(z)=z_1$. An $I$ basis can hence alternatively be given by a collection of $\{x_g, y_g, z_g\}_{g \in G}$ satisfying $x_g y_{xg} = y_g x_{yg}, x_g z_{xg} = z_g x_{zg},y_g z_{yg} = z_g y_{zg}$
\end{remark}
In \cite[\S 4]{TateVdB} Tate and Van den Bergh give a construction for an $I$-basis for a Sklyanin algebra. In the case of a quadratic Sklyanin algebra this construction depends on the choice of a rational point $\overline{o} = (o_1,o_2,o_3) \in Y^3$. For each $g \in G$ one defines $\overline{og}$ by setting
\begin{eqnarray*}
 \overline{ox} & = & (\psi o_1, \psi^{-2} o_2, \psi^{-2} o_3) \\
 \overline{oy} & = & (\psi^{-2} o_1, \sigma o_2, \psi^{-2} o_3) \\
 \overline{oz} & = & (\psi^{-2} o_1, \psi^{-2} o_2, \psi o_3) \\
\end{eqnarray*}
such that if $g=x^\alpha y^\beta z^\lambda$ then $\overline{og} = (\psi^{\alpha-2\beta-2\lambda} o_1, \psi^{\beta-2\alpha-2\lambda} o_2, \psi^{\lambda-2\alpha-2\beta} o_3)$.\\

We then define $x_g, y_g, z_g \in A_1 = \Gamma(Y,\Lscr)$ (up to a scalar multiple) by setting
\begin{eqnarray*} x_g((\overline{og})_1) \neq 0 & x_g((\overline{og})_2) = 0 & x_g((\overline{og})_3) = 0 \\
y_g((\overline{og})_1) = 0 & y_g((\overline{og})_2) \neq 0 & y_g((\overline{og})_3) = 0 \\
z_g((\overline{og})_1) = 0 & z_g((\overline{og})_2) = 0 & z_g((\overline{og})_3) \neq 0
\end{eqnarray*}
(the scalar multiples are then chosen such that the relations in Remark \ref{rem:scalarmultiple} hold)\\

In particular
\begin{eqnarray*}
v(x)=x_1 \in \Gamma(Y,\Lscr(-o_2-o_3)) & = & \Hom_X(\Oscr_X,\Oscr_X(1) \otimes m_{o_2+o_3})\\
& = & \Hom_X(\Oscr_X, (\Oscr_X \otimes m_{\sigma^{-1}o_2+\sigma^{-1}o_3})(1))
\end{eqnarray*}
and analogously 
\begin{eqnarray*} v(y) \in \Hom_X(\Oscr_X,(\Oscr_X \otimes m_{\psi^{-1}o_1+\psi^{-1}o_3})(1) ) \\ v(z) \in \Hom_X(\Oscr_X,(\Oscr_X  \otimes m_{\psi^{-1}o_1+\psi^{-1}o_2})(1)) \end{eqnarray*}

Similar computations are possible for monomials of higher degree, for example: $v(xy) =x_1 \cdot y_x$ lies in the image of
\[ \Hom_X(\Oscr_X,(\Oscr_X \otimes m_{\psi^{-1}o_2+ \psi^{-1}o_3})(1)) \otimes \Hom_X(\Oscr_X,(\Oscr_X \otimes m_{\psi^{-1}\overline{ox}_1+ \psi^{-1}\overline{ox}_3})(1)) \hookrightarrow A_1 \otimes A_1 \rightarrow A_2 \]
This image is given by 
\begin{eqnarray*} \Hom\left( \Oscr_X, \left( \Oscr_X \otimes m_{\psi^{-1}o_2+ \psi^{-1} o_3} m_{\psi^{-1} o_1 + \psi^{-4} o_3} \right) (2) \right) \\
\cong \Hom\left( \Oscr_X, \left( \Oscr_X \otimes m_{\psi^{-1} o_1 + \psi^{-1}o_2+ \psi^{-1} o_3} m_{\tau^{-1}\psi^{-1} o_3} \right) (2) \right)  \end{eqnarray*}
Where we used the fact that $o_1, o_2, o_3$ lie in different $\tau$-orbits. Inspired by the above results we make the following choices for $o_1, o_2, o_3$:
\begin{eqnarray} \label{eq:pointso} o_1 = \sigma p \textrm{ and } o_2 = \sigma q \textrm{ and }o_3 = \sigma r \end{eqnarray}
Where $r$ is some point on $Y$ lying in a different $\tau$-orbit than $p$ and $q$.
(recall that we required $p$ and $q$ to lie in different $\tau$-orbits).

We can can inductively show the following: assume $\alpha, \beta , \lambda$ are nonnegative integers then:
\begin{eqnarray} \notag v(x^\alpha y^\beta z^\lambda) \in \Hom \Big( \Oscr_X , \big( \Oscr_X & \otimes & m_p m_{\tau^{-1}p} \ldots m_{\tau^{-\beta - \gamma +1} p} \\
 \label{eq:vliesin} & & m_q  m_{\tau^{-1}q} \ldots m_{\tau^{-\alpha - \gamma + 1} q} \\
\notag & & m_r m_{\tau^{-1}r} \ldots  m_{\tau^{-\alpha - \beta + 1} r}\big)(\alpha + \beta + \gamma) \Big) \end{eqnarray}

\subsection{Proof of Lemma  \ref{lem:negativedim}}
Using the above language of I-bases the following proposition reduces the proof of \eqref{eq:Ibasisproof} to a combinatorial problem
\begin{proposition} \label{prp:updownimplication}
Let $\alpha, \beta, \gamma \geq 0$ then
\begin{eqnarray} \notag v(x^\alpha y^\beta z^\lambda) & \in &  \Hom \Big( \Oscr_X , \big( \Oscr_X \otimes m_{d_0} \ldots m_{d_{h-1}} \big)(\alpha + \beta + \gamma) \Big) \\
\label{eq:implication} &  \Updownarrow & \\
\notag \left \lceil \frac{h}{2} \right \rceil \leq \beta + \gamma & \textrm{and} & \left \lfloor \frac{h}{2} \right \rfloor \leq \alpha + \gamma
 \end{eqnarray}
\end{proposition}
\begin{proof}
$\Uparrow$ follows from $(\ref{eq:vliesin})$. For $\Downarrow$ we need some more computations ... 
\end{proof}
Fix $h \in \mathbb{N}$ and define the following right submodules $M$ and $M'$ of $A_A$:
\begin{eqnarray*}
M_n & = & \Hom_X \left( \Oscr_X, \left( \Oscr_X \otimes m_{d_0} \ldots m_{d_{h-1}} \right) (n) \right) \\
M'_n & = & \textrm{Span} \{ v(x^\alpha y^\beta z^\lambda) \mid \alpha + \beta + \lambda = n, \lceil h/2 \rceil \leq \beta + \gamma \textrm{ and } \lfloor h/2 \rfloor \leq \alpha + \gamma \}
\end{eqnarray*}
To see that $M'$ is a right $A$-module, recall that for each $g = x^\alpha y^\beta z^\lambda$, the elements $x_g, y_g, z_g$ give a $k$-basis for $A_1$ hence the image of
\[ k \cdot v(x^\alpha y^\beta z^\lambda) \otimes A_1 \rightarrow A_n \otimes A_1 \rightarrow A_{n+1} \]
lies inside $\textrm{Span} \{ v(x^{\alpha+1} y^\beta z^\lambda), v(x^\alpha y^{\beta+1} z^\lambda) , v(x^\alpha y^\beta z^{\lambda+1}) \}$.\\
We then have the following lemmas.
\begin{lemma} \label{lem:Ibasis1}
Let $M$ and $M'$ be as above then for all $n$ sufficiently large we have 
\[ \dim_k (M_n) = \dim_k(M'_n) \]
\end{lemma}
\begin{proof}
We prove this for $h$ even. The case $h$ odd is completely similar. Let $h=2a$, then we have for $n \geq 2a-1$ we have 
\[ \dim_k (M_n) = \frac{(n+1)(n+2)}{2} - 2 \cdot \frac{a(a+1)}{2} = \dim_k(M'_n) \]
for $M$ this follows from Lemma \ref{cor:standardvanishing} and \cite[Corollary 5.2.4]{VdB19}. For $M'$ this follows from the fact that for $n \geq 2a-1$ at most one of the inequalities $\alpha + \lambda \geq a, \beta + \lambda \geq a \}$ can fail when $\alpha + \beta + \lambda = n$.
\end{proof}
\begin{lemma} \label{lem:Ibasis2}
Let $N$ be a right $A$-modules such that $M' \subset N \subset A$ and suppose there is an $n_0 \in \mathbb{N}$ such that $M'_{n_0} \subsetneq N_{n_0}$ then we have $M'_n \subsetneq N_n$ for all $n \geq n_0$
\end{lemma}
\begin{proof}
By induction it suffices to show  $M'_{n_0 +1} \subsetneq N_{n_0 +1}$. For this choose some nonzero element in $N_{n_0} \setminus M'_{n_0}$. This element can be written as $\displaystyle \sum_{\substack{\alpha, \beta, \lambda \\ \alpha + \beta + \lambda =n}} t_{ \alpha, \beta, \lambda } v(x^\alpha y^\beta z^\lambda)$. Without loss of generality we can assume there is a $t_{\alpha_0, \beta_0, \lambda_0} \neq 0$ with $\displaystyle \beta_0 + \lambda_0 < \left \lceil \frac{h}{2} \right \rceil$. Choose such a $t$ with $\alpha_0$ maximal and let $g = x^{\alpha_0}y^{\beta_0}z^{\lambda_0}$ then
\[ \left( \sum_{\substack{\alpha, \beta, \lambda \\ \alpha + \beta + \lambda =n}} t_{ \alpha, \beta, \lambda } v(x^\alpha y^\beta z^\lambda) \right) x_g  = t_{ \alpha_0, \beta_0, \lambda_0 } v(x^{\alpha_0+1} y^{\beta_0} z^{\lambda_0}) + \sum_{\substack{\alpha', \beta', \lambda' \\ \alpha' + \beta' + \lambda' =n+1 \\
\alpha' \leq \alpha_0}} t'_{ \alpha', \beta', \lambda'} v(x^{\alpha'} y^{\beta'} z^{\lambda'}) \]
is a nonzero element in $N_{n_0+1} \setminus M'_{n_0+1}$
\end{proof}
\begin{proof}[continuation of Proposition \ref{prp:updownimplication}]
We have already proven $\Downarrow$ which is equivalent to $M' \subset M$. It then immediately follows from Lemmas \ref{lem:Ibasis1} and \ref{lem:Ibasis2} that $M = M'$, finishing the proof of the proposition.
\end{proof}

We can now prove $(\ref{eq:Ibasisproof})$. By Proposition \ref{prp:updownimplication} it suffices to count the number of triples of natural numbers $\alpha, \beta, \lambda$ satisfying $\alpha + \beta + \lambda = a+b$ and $\left \lceil \frac{b}{2} \right \rceil \leq \beta + \gamma$ and $ \left \lfloor \frac{b}{2} \right \rfloor \leq \alpha + \gamma$. The latter is equivalent to $\alpha \leq a+b-\left \lceil \frac{b}{2} \right \rceil, \beta \leq a+b - \left \lfloor \frac{b}{2} \right \rfloor$ such that our problem has turned into a combinatorial problem: we need to show

%Let $o_1 = \sigma^{-a-2b}p, o_2 = \sigma^{-a-2b}q, o_3 = \sigma^{-a-2b}r$ then a $k$-basis for $A_{a+2b} = \Hom(\Oscr_X, \Oscr_X(a+2b) )$ is given by $\{ v(x^\alpha y^\beta z^\lambda) \}_{\alpha + \beta + \lambda = a + 2b}$. By $(\ref{eq:vliesin})$ such a basis element lies in $\Hom(\Oscr_X, \Oscr_X(a+2b) \otimes m_{d} \otimes \ldots \otimes m_{\tau^{-b+1}d} )$ if (AND ONLY IF???) $\min{\alpha + \beta, \alpha + \lambda, \beta + \lambda} \geq b$ which is equivalent to $\max{ \alpha, \beta , \lambda } \leq a + b$. Proving $(\ref{eq:Ibasisproof})$ has hence turned into a combinatorial problem: we need to calculate
\begin{lemma} \label{prp:Ibasiscomb}
\begin{eqnarray}
\notag \# \Big \{ (\alpha, \beta, \lambda) \in \mathbb{N}^3 & \mid & \alpha + \beta + \lambda = a+ b, \\ 
\label{eq:combinatorial} & & \alpha \leq a+b-\left \lceil \frac{b}{2} \right \rceil, \beta \leq a+b - \left \lfloor \frac{b}{2} \right \rfloor \Big \} = h'(2a+b) 
\end{eqnarray}
\end{lemma}
We first show that the left hand side equals zero when $2a+b<0$. Note that in this case $a+b<\frac{b}{2} \leq \left \lceil \frac{b}{2} \right \rceil$, hence the condition $\alpha \leq a+b-\left \lceil \frac{b}{2} \right \rceil$ contradicts $\alpha \in \mathbb{N}$ such that the left hand side of \eqref{eq:combinatorial} is zero. Hence from now on we can assume $2a+b \geq 0$.\\

%TO DO: work from here on !!!

This combinatorial problem has a graphical interpretation: it asks for counting the number of dots in Figure \ref{fig:dots1} whose coefficients $(\alpha,\beta,\lambda)$ satisfy the above inequalities. %with all 3 coefficients at most $a+b$ in the following grid:
\begin{figure}[h]
\includegraphics[height=6cm]{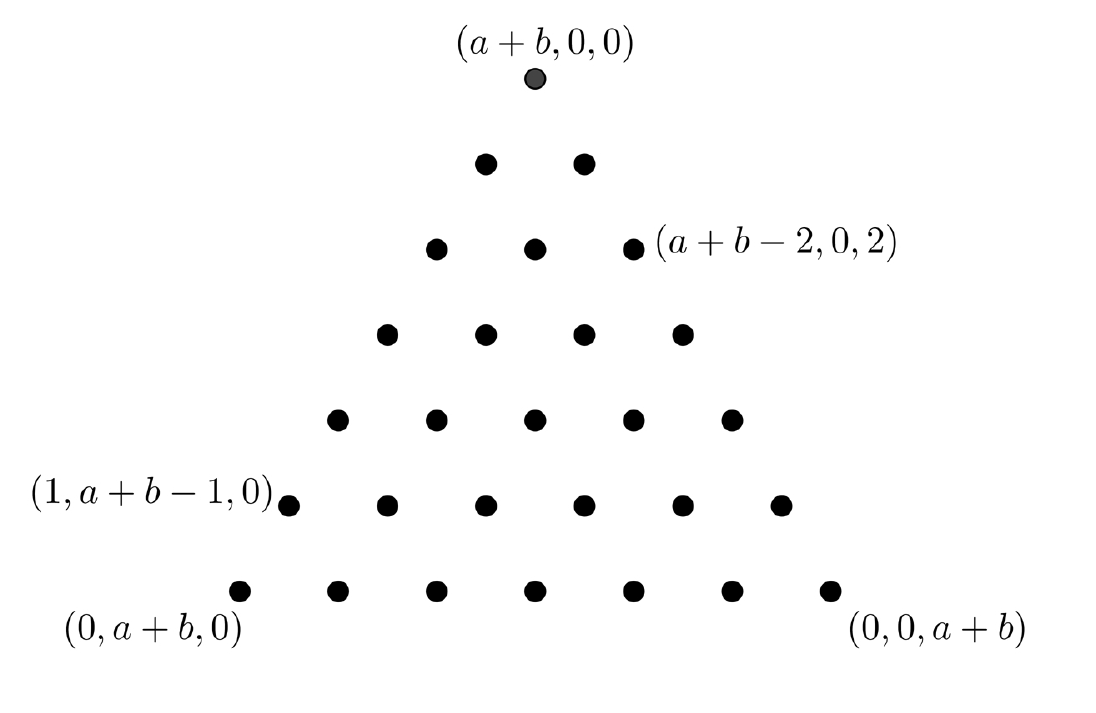}
\caption{Combinatorial problem} \label{fig:dots1}
\end{figure}

In order to compute this number of dots we can rewrite $(\ref{eq:combinatorial})$ as:
\begin{eqnarray}
\notag \# \Big \{ (\alpha, \beta, \lambda) \in \mathbb{N}^3 \mid \alpha + \beta + \lambda = a+ b, \alpha \leq a+b-\left \lceil \frac{b}{2} \right \rceil, \beta \leq a+b - \left \lfloor \frac{b}{2} \right \rfloor \Big \} & = & \\
 \notag \# \Big\{ (\alpha, \beta, \lambda) \in \mathbb{N}^3 \mid \alpha + \beta + \lambda = a+ b \Big\} - & & \\
\label{eq:calculations} \textcolor{LimeGreen}{  \# \Big \{ (\alpha, \beta, \lambda) \in \mathbb{N}^3 \mid \alpha + \beta + \lambda = a+ b, \alpha \geq a+b+1-\left \lceil \frac{b}{2} \right \rceil \Big\} } - & & \\
\notag \textcolor{Red}{\# \Big \{ (\alpha, \beta, \lambda) \in \mathbb{N}^3 \mid \alpha + \beta + \lambda = a+ b, \beta \geq a+b+1 - \left \lfloor \frac{b}{2} \right \rfloor \Big \}} + & & \\
\notag \textcolor{Blue}{\# \Big \{ (\alpha, \beta, \lambda) \in \mathbb{N}^3 \mid \alpha + \beta + \lambda = a+ b, \alpha \geq a+b+1-\left \lceil \frac{b}{2} \right \rceil, \beta \geq a+b+1 - \left \lfloor \frac{b}{2} \right \rfloor \Big \} }
\end{eqnarray}

The green, red and blue numbers can be visualized as in Figure \ref{fig:dots2}. %Essentially there are 3 options: 1) the green areas do not overlap, 2) the green areas overlap in red areas, but the red areas do not overlap and 3) the red areas overlap in a blue area. \\
\begin{figure}[h]
\includegraphics[height=6cm]{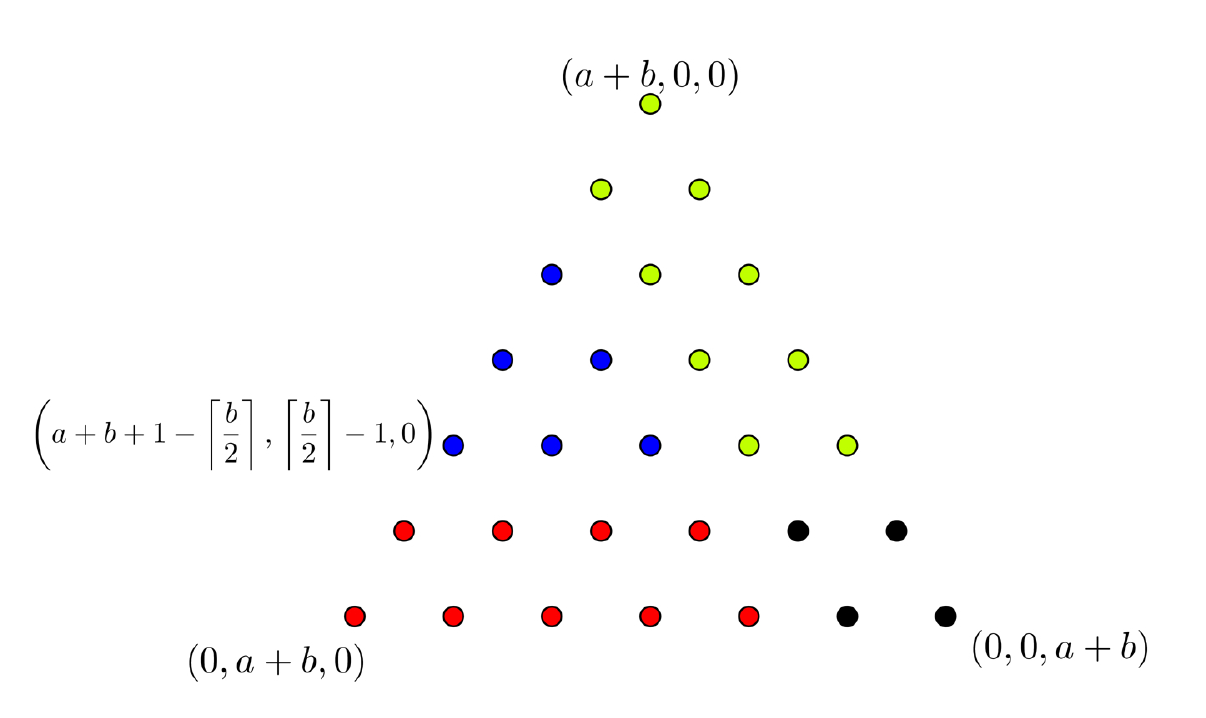}
\caption{Visualizing formula \eqref{eq:calculations}} \label{fig:dots2}
\end{figure}

The reason for writing our combinatorial problem as in $(\ref{eq:calculations})$ is the existence of the following bijection:
\begin{eqnarray*} \{ (\alpha, \beta, \lambda) \in \mathbb{N}^3 \mid \alpha + \beta + \lambda = n_1 , \alpha \geq n_2 \} & \rightarrow & \{ (\alpha, \beta, \lambda) \in \mathbb{N}^3 \mid \alpha + \beta + \lambda = n_1-n_2 \}:\\ (\alpha, \beta, \lambda) & \mapsto & (\alpha-n_2, \beta, \lambda) \end{eqnarray*}
Using similar bijections for the other sets we can write $(\ref{eq:calculations})$ as
\begin{eqnarray}
\notag \# \Big \{ (\alpha, \beta, \lambda) \in \mathbb{N}^3 \mid \alpha + \beta + \lambda = a+ b, \alpha \leq a+b-\left \lceil \frac{b}{2} \right \rceil, \beta \leq a+b - \left \lfloor \frac{b}{2} \right \rfloor \Big \} & = & \\
\notag \# \Big\{ (\alpha, \beta, \lambda) \in \mathbb{N}^3 \mid \alpha + \beta + \lambda = a+ b \Big\} - & & \\
 \label{eq:calculations2} \textcolor{LimeGreen}{  \# \Big \{ (\alpha, \beta, \lambda) \in \mathbb{N}^3 \mid \alpha + \beta + \lambda = \left \lceil \frac{b}{2} \right \rceil -1 \Big\} } - & & \\
 \notag \textcolor{Red}{\# \Big \{ (\alpha, \beta, \lambda) \in \mathbb{N}^3 \mid \alpha + \beta + \lambda = \left \lfloor \frac{b}{2} \right \rfloor -1 \Big \}} + & & \\
 \notag \textcolor{Blue}{\# \Big \{ (\alpha, \beta, \lambda) \in \mathbb{N}^3 \mid \alpha + \beta + \lambda = \left \lceil \frac{b}{2} \right \rceil + \left \lfloor \frac{b}{2} \right \rfloor -a-b-2 = -a-2 \Big \} }
\end{eqnarray}
Now we can use the following: for all $n \geq 0$:
\begin{eqnarray}
\notag \# \{ (\alpha, \beta, \lambda) \in \mathbb{N}^3 \mid \alpha + \beta + \lambda = n \} & = & \sum_{i=0}^n \# \{ (\alpha, \beta) \in \mathbb{N}^2 \mid \alpha + \beta = i \} \\
\label{eq:calculations3} & = & \sum_{i=0}^n i+1 \\
\notag & = & \frac{(n+2)(n+1)}{2}
\end{eqnarray}
(recall that we assumed $a \leq -2$ and $a+b \geq 0$ such that $\displaystyle -a-2, \left \lceil \frac{b}{2} \right \rceil -1,  \left \lfloor \frac{b}{2} \right \rfloor-1 \geq 0$)
Hence combining \eqref{eq:calculations2} and \eqref{eq:calculations3} we find that \[ \# \Big \{ (\alpha, \beta, \lambda) \in \mathbb{N}^3 \mid \alpha + \beta + \lambda = a+ b, \alpha \leq a+b-\left \lceil \frac{b}{2} \right \rceil, \beta \leq a+b - \left \lfloor \frac{b}{2} \right \rfloor \Big \} \]
equals
\begin{eqnarray}
\notag \frac{(a+b+2)(a+b+1)}{2} - \frac{\left\lceil \frac{b}{2} \right \rceil \left( \left \lceil \frac{b}{2} \right \rceil +1 \right)}{2} - \frac{\left \lfloor \frac{b}{2} \right \rfloor \left( \left \lfloor \frac{b}{2} \right \rfloor +1 \right)}{2} + \frac{(-1-a)(-a)}{2} & = & \\
\label{eq:calculations4} \frac{a^2+2ab+b^2+3a+3b+2-\left \lceil \frac{b}{2} \right \rceil ^2 - \left \lfloor \frac{b}{2} \right \rfloor^2 -\left \lceil \frac{b}{2} \right \rceil - \left \lfloor \frac{b}{2} \right \rfloor +a^2+a}{2} & = & \\
\notag \frac{2a^2+2ab+b^2+4a+2b+2-\left \lceil \frac{b}{2} \right \rceil ^2 - \left \lfloor \frac{b}{2} \right \rfloor^2 }{2} & &
\end{eqnarray}
We now treat the cases $b$ even and $b$ odd separately. First assume $b=2r$ for some $r \in \mathbb{N}$. Then \eqref{eq:calculations4} equals
\[ \frac{2a^2+4ar+4r^2+4a+4r+2- 2 r^2 }{2} = a^2+2ar+r^2+2a+2r+1 = (a+r+1)^2 \] 
Next assume $b=2r+1$. Then \eqref{eq:calculations4} equals
\begin{eqnarray*} \frac{2a^2+4ar+2a+4r^2+4r+1+4a+4r+2+2- (r+1)^2 - r^2 }{2} & = & \\
\frac{2a^2+4ar+2r^2+6r+6a+4}{2} & = & \\
a^2+2ar+r^2+3r+3a+2 & = & \\
 (a+r+1)(a+r+2) & & \end{eqnarray*}
Finally letting $n=a+r$ we see that this agrees with \eqref{eq:h2a+b}
\bibliographystyle{amsplain}
\bibliography{Symreferences}

\end{document}